\documentclass{amsart}
\usepackage{amsmath,amsthm,amssymb,amsfonts}
\usepackage[utf8]{inputenc}
\usepackage[colorlinks]{hyperref}
\usepackage{verbatim}
\usepackage{tikz}
\usepackage{tikz-cd}
\usepackage[colorinlistoftodos]{todonotes}
\usepackage{pgf}
\usepackage{cancel}
\newtheorem{thm}{Theorem}[section]
\newtheorem{lem}[thm]{Lemma}

\newtheorem{prop}[thm]{Proposition}
\theoremstyle{remark}
\newtheorem{rem}[thm]{Remark}

\newtheorem{claim}{Claim}[thm]

\newcommand{\Blank}{\mathsf{B}}
\newcommand{\Terminal}{\mathsf{T}}
\newcommand{\Cross}{\mathsf{C}}
\newcommand{\NCross}{\bcancel{\mathsf{C}}}

\newcommand{\Leftturn}{\mathsf{L}}
\newcommand{\Ordinary}{\mathsf{O}}
\newcommand{\DCross}{\mathsf{D}}
\newcommand{\Straight}{\mathsf{S}}
\newcommand{\Initial}{\mathsf{I}}
\newcommand{\NPlus}{\bcancel{\mathsf{P}}}

\newcommand{\Plus}{\mathsf{P}}

\newcommand{\Groth}{\mathfrak{G}}
\newcommand{\Schub}{\mathfrak{S}}
\newcommand{\Gb}{\mathfrak{G}^{(\beta)}}
\newcommand{\id}{\mathrm{id}}
\newcommand{\MBPD}{\mathrm{MBPD}}
\newcommand{\PD}{\mathrm{PD}}

\newcommand{\RCP}{\mathrm{RCP}}

\newcommand{\wt}{\mathrm{wt}}
\newcommand{\Z}{\mathbb{Z}}

\DeclareMathOperator{\rpop}{\mathrm{rpop}}
\newcommand{\es}{e^*}
\newcommand{\fs}{f^*}
\newcommand{\ftarg}{F}
\newcommand{\etarg}{E}
\newcommand{\window}{window}
\newcommand{\Y}{Y}

\newcommand{\XX}{X}
\newcommand{\II}{I}

\newcommand{\term}{\mathrm{T}}
\newcommand{\nonterm}{{\mathrm{\overline{T}}}}

\newcommand{\btile}{
 \begin{tikzpicture}[x=1em,y=1em,thick,color = blue]
\draw[step=1,gray,thin] (0,0) grid (1,1);
\draw[color=black, thick, sharp corners] (0,0) rectangle (1,1);
\end{tikzpicture}}

\newcommand{\htile}{
\begin{tikzpicture}[x=1em,y=1em,thick,color = blue]
\draw[step=1,gray,thin] (0,0) grid (1,1);
\draw[color=black, thick, sharp corners] (0,0) rectangle (1,1);
\draw(1.0,0.5)--(0.0,0.5);
\end{tikzpicture}
}

\newcommand{\vtile}{
\begin{tikzpicture}[x=1em,y=1em,thick,color = blue]
\draw[step=1,gray,thin] (0,0) grid (1,1);
\draw[color=black, thick, sharp corners] (0,0) rectangle (1,1);
\draw(0.5,1.0)--(0.5,0.0);
\end{tikzpicture}}

\newcommand{\ptile}{
\begin{tikzpicture}[x=1em,y=1em,thick,color = blue]
\draw[step=1,gray,thin] (0,0) grid (1,1);
\draw[color=black, thick, sharp corners] (0,0) rectangle (1,1);
\draw(0.5,1.0)--(0.5,0.0);
\draw(1.0,0.5)--(0.0,0.5);
\end{tikzpicture}}

\newcommand{\rtile}{
\begin{tikzpicture}[x=1em,y=1em,thick,color = blue]
\draw[step=1,gray,thin] (0,0) grid (1,1);
\draw[color=black, thick, sharp corners] (0,0) rectangle (1,1);
\draw(1.0,0.5)--(0.5,0.5)--(0.5,0.0);
\end{tikzpicture}}

\newcommand{\jtile}{
\begin{tikzpicture}[x=1em,y=1em,thick,color = blue]
\draw[step=1,gray,thin] (0,0) grid (1,1);
\draw[color=black, thick, sharp corners] (0,0) rectangle (1,1);
\draw(0.5,1.0)--(0.5,0.5)--(0.0,0.5);
\end{tikzpicture}}

\newcommand{\mtile}{
\begin{tikzpicture}[x=1em,y=1em,thick,color = blue]
\draw[step=1,gray,thin] (0,0) grid (1,1);
\draw[color=black, thick, sharp corners] (0,0) rectangle (1,1);
\draw(0.5,1.0)--(0.5,0.5)--(0.0,0.5);
\node at (0.5,0.5) {$\bullet$};
\end{tikzpicture}}

\newcommand{\bumptile}{
\begin{tikzpicture}[x=1em,y=1em,thick,rounded corners,color = blue]
\draw[step=1,gray,thin] (0,0) grid (1,1);
\draw[color=black, thick, sharp corners] (0,0) rectangle (1,1);
\draw(1.0,0.5)--(0.5,0.5)--(0.5,0.0);
\draw(0.5,1.0)--(0.5,0.5)--(0.0,0.5);
\end{tikzpicture}}

\title{Marked Bumpless Pipedreams and Compatible Pairs}

\author{Daoji Huang}
\author{Mark Shimozono}
\author{Tianyi Yu}

\begin{document}

\begin{abstract}
    We construct a bijection between marked bumpless pipedreams with reverse compatible pairs, which are in bijection with not-necessarily-reduced pipedreams. This directly unifies various formulas for Grothendieck polynomials in the literature. Our bijection is a generalization of a variant of the bijection of Gao and Huang in the unmarked, reduced case.
\end{abstract}

\maketitle

\section{Introduction}

The polynomial ring $\mathbb{Z}[\beta][x_1,x_2,\cdots]$ admits operators
\[\partial_i:=\frac{1-s_i}{x_i-x_{i+1}}\text{ and }\pi_i:= \partial_i(1+\beta x_{i+1}),\]
where $s_i$ acts by switching $x_i$ and $x_{i+1}$, $x_i$ acts as multiplication by $x_i$, and $1$ denotes the identity operator. The \emph{$\beta$-Grothendieck polynomials} $\Gb_w$ for $w\in S_\infty:=\bigcup_{i=1}^\infty S_n$ are the unique family of polynomials that satisfy
\begin{align*}
\Gb_{w_0^n}&=x_1^{n-1}x_2^{n-2}\cdots x_{n-1}&\quad&\text{where
$w_0^n:=n\ n-1 \ \cdots \ 1$} \\
\pi_i\Gb_w&=\pi_i\Gb_{ws_i}&&\text{if $\ell(ws_i)=\ell(w)-1$.}
\end{align*}
The $\beta$-Grothendieck polynomials indexed by $w\in S_n$ 
are a Schubert basis of connective $K$-theory of the flag variety \cite{H}. 
Specializing $\beta$ yields two important families of polynomials, both of which were introduced in this form by Lascoux and Sch\"utzenberger \cite{LS}.
The Grothendieck polynomials $\Groth_w = \Gb_w|_{\beta=-1}$ represent structure sheaves of Schubert varieties in the flag variety. The Schubert polynomials $\Schub_w = \Gb_w|_{\beta=0}$ of \cite{LS} represent the cohomology classes of Schubert varieties in the flag variety \cite{BGG} \cite{Dem} \cite{LS}.
Equivalently $\Schub_w$ is the coefficient of $\beta^0$ in $\Gb_w$ or equivalently the lowest $x$-degree component of $\Gb_w$.

The $\beta$-Grothendieck polynomials satisfy the positivity
$\Gb_w\in \Z_{\ge0}[\beta;x_1,x_2,\dotsc]$ for all $w\in S_\infty$.
Due to this positivity, any combinatorial formula for $\Gb_w$ based on a bijection between a set and the monomials of $\Gb_w$, restricts to a combinatorial formula for 
Schubert polynomials.

There are many combinatorial formulas for Grothendieck polynomials. 
Fomin and Kirillov \cite[Prop. 3.3]{FK} gave a compatible pair formula 
for Grothendieck polynomials, extending the formula of Billey, Jockusch, and Stanley \cite{BJS}
for Schubert polynomials. Knutson and Miller \cite{KM} interpreted Grothendieck polynomials as $K$-polynomials of matrix Schubert varieties (up to change of variables) and showed via Gr\"obner degeneration they can be computed with (not-necessarily-reduced) pipedreams, generalizing the reduced pipedream formula of Billey and Bergeron \cite{BB} for Schubert polynomials. Knutson and Miller \cite{KM} also showed that
the Schubert polynomials are multidegrees of matrix Schubert varieties.
Billey and Bergeron \cite{BB} gave a bijection between reduced compatible pairs and reduced pipedreams and their formula generalizes directly to the not-necessarily-reduced case. On the other hand, using marked bumpless pipedreams, Weigandt \cite{W} reinterpreted Lascoux's formulas \cite{Las} for Grothendieck polynomials based on alternating sign matrices (ASMs),
which extends the bumpless pipedream formula of \cite{LLS} for Schubert polynomials.

Gao and Huang \cite{GH} gave a bijection between reduced pipedreams and reduced bumpless pipedreams.
The goal of this paper is to extend the bijection of Gao and Huang, furnishing a bijection
between pipedreams and marked bumpless pipedreams.
This provides a direct connection between the various kinds of combinatorial formulas for Grothendieck polynomials in the literature. In a later paper we shall prove that our bijection restricts to that of \cite{GH} in the reduced case. 
The Gao--Huang bijection processes tiles in a bumpless pipedream from top to bottom
using a generalization of the column moves of \cite{LLS} which operate on adjacent columns. In particular, the generalization involves scanning over blank tiles. 
In the reduced case our map processes tiles from bottom to top,
has no blank scanning, and operates on adjacent rows by row moves, 
the transpose of \cite{LLS} column moves. 

In fact, long before the discovery of bumpless pipedreams \cite{LLS}, it was known that the 2-enumeration of alternating sign matrices of size $n$ is $2^{n(n-1)/2}$ \cite{MPR,EKLP,Kuperberg}, which is clearly the number of not-necessarily reduced pipedreams of size $n$. Therefore, based on Weigandt's observation that connects alternating sign matrices and bumpless pipedreams \cite{W}, our bijection can be viewed as a bijection between 2-enumerated ASMs and not-necessarily reduced pipedreams. This connection between ASMs and pipedreams was recently explored by Striker and Huang \cite{HS} to establish a partial bijection between totally symmetric self-complementary plane partitions and ASMs, using the Gao--Huang bijection.

The not-necessarily-reduced setting introduces many complications which make it considerably more
subtle than the reduced case. 

\subsection{Pipedreams} 
Let $[n]=\{1,2,\dotsc,n\}$. A (not-necessarily-reduced) pipedream is a tiling of $\{(i,j)\in[n]\times[n]: i+j\le n+1\}$, such that all the entries $(i,j)$ with $i+j=n+1$ have tiles $\jtile$, and all other entries are either $\ptile$ or $\bumptile$. 
We use the matrix-style notation $D_{i,j}$ for the tile in the $i$-th row and $j$-th column.
Every such diagram can be viewed as $n$ pipes entering from the top and exit from the left. We denote the set of pipedreams in the $n\times n$ grid with $\PD(n)$. Each $D\in \PD(n)$ has an associated permutation $w=w(D)\in S_n$ where the pipe entering at the left border at the $i$-th row and exits the top border at the $w(i)$-the column, with the proviso that if two pipes have already crossed, the subsequent crossings between this pair of pipes are ignored (that is, treated like $\bumptile$.). For $w\in S_n$, let $\PD(w)$ denote the set of pipedreams with associated permutation $w$. For $w\in S_n$, the Grothendieck polynomial $\mathfrak{G}_w$ can be computed via pipedreams \cite{FK}:
\begin{align}\label{E:G pipedream}
\Gb_w =\sum_{D\in\PD(w)} \beta^{|\mathrm{cross}(D)|-\ell(w)}\prod_{(i,j)\in\mathrm{cross}(D)}x_i
\end{align}
where $\mathrm{cross}(D)$ denote the set of coordinates of the $\ptile$ in $D$.

\subsection{Compatible pairs}
A \textit{biletter} is an ordered pair of integers $(i,a)$ with $1\le i\le a < n$.
A \textit{reverse compatible pair} is a sequence of biletters
\begin{align}\label{E:RCP}
B = ((i_1,a_1),(i_2,a_2),\dotsc,(i_\ell,a_\ell))
\end{align}
which are strictly decreasing in the order given by $(i_1,a_1)>(i_2,a_2)$ if $i_1>i_2$ or if $i_1=i_2$ and $a_1<a_2$. 
Let $\RCP(n)$ denote the set of reverse compatible pairs.

The weight of $B$ is defined by $\wt(B)=(m_1,\dotsc,m_n)\in\Z_{\ge0}^n$ where $m_i$ is the number of biletters in $B$ of the form $(i,a)$ for some $a$.

The Demazure or $0$-Hecke product $*$ on permutations is the unique monoid structure such that for any simple reflection $s_i$ and permutation $w$, 
\begin{equation*}
    s_i * w = \begin{cases}
        s_i w & \text{if $s_i w>w$} \\
        w & \text{otherwise.}
    \end{cases}
\end{equation*}
Every $B\in \RCP(n)$ has an associated permutation
$w(B) = s_{a_\ell} * s_{a_{\ell-1}} * \dotsm * s_{a_1}$
with notation as in \eqref{E:RCP}. 
Note that the subscripts of $a$ are \emph{decreasing}.

Denote by $\RCP(w)$ the set of $B\in \RCP(n)$ with associated permutation $w$. Let $|B|$ be the length of the sequence $B$.

\begin{rem} \label{R:CP and PD} 
There is a bijection $\RCP(n)\to \PD(n)$ that sends the compatible pair $B=((i_1,a_1),\dotsc,(i_\ell,a_\ell))$ to the pipedream with crossings at positions $(i_j, a_j - i_j + 1)$ for $1\le j\le \ell$.
This was proved in \cite{BB} for the reduced case but the proof works in the nonreduced setting. 
This bijection preserves the weight and the associated permutation. Therefore 
\begin{align}\label{E:Groth by RCP}
  \Gb_w &= \sum_{B\in \RCP(w)} \beta^{|B|-\ell(w)} x^{\wt(B)}.
\end{align}
\end{rem}

\subsection{Marked Bumpless pipedreams}
We work with a set of tiles
\[
\btile,\quad \htile, \quad \vtile,\quad \ptile,\quad,\rtile,\quad\jtile,\quad\mtile
\]
named blank, horizontal, vertical, plus, R, J, and marked J.

We will say that a tile \textit{connects to the right} (resp. left, up, down) if it contains a line segment going from its center to the right (resp. left, up, down). 

A \textit{marked bumpless pipedream} is an
$[n]\times[n]$ matrix with entries in the above set of tiles
such that every row has a pipe entering from the right, every column has a pipe leaving to the 
south, no pipe enters from the top, no pipe leaves to the left, and the ``pipes are connected".
Let $\MBPD(n)$ be the set of $[n]\times[n]$ marked bumpless pipedreams. 
Each $D\in \MBPD(n)$ has an \textit{associated permutation} $w=w(D) \in S_n$ where the pipe entering the right border at the $i$-th row, exits the bottom border in the $w(i)$-th column, with the proviso that if two pipes have already crossed, then subsequent crossings between this pair of pipes are ignored, that is, treated like bump tiles $\bumptile$.
For $w\in S_n$ let $\MBPD_w = \{ D\in \MBPD(n)\mid w(D)=w\}$, the bumpless pipedreams with associated permutation $w$. For every $w\in S_n$ the \textit{Rothe bumpless pipedream} 
$D_w$ is the unique element of $\MBPD_w$ whose $i$-th pipe turns only at position $(i,w(i))$ for all $i$.

Say that a tile is \textit{heavy} if it is either a $\btile$ or a $\mtile$ and \textit{light} otherwise. The \textit{weight} of $D\in\MBPD(n)$ is the sequence $\wt(D)=(m_1,m_2,\dotsc,m_n)\in\Z_{\ge0}^n$ where $m_i$ is the number of heavy tiles in the $i$-th row of $D$. The only element of $\MBPD(n)$ with no heavy tile is $D_\id$.

Our main theorem is a constructive proof of the following.

\begin{thm}
\label{T: main}
There are mutually-inverse bijections
$\Phi:\MBPD(n)\to \RCP(n)$ and $\Psi:\RCP(n)\to\MBPD(n)$ which preserve weights and associated permutations. Furthermore, these bijections restrict to mutually-inverse bijections between reduced, unmarked bumpless pipedreams and reduced reverse compatible sequences.
\end{thm}

\section{Terminology on tiles}
\subsection{Tile notation}
Given $D\in \MBPD(n)$ and intervals $I,J\subset [n]$ we denote by $D_{I,J}$ the submatrix of entries $D_{i,j}$ for $i\in I$ and $j\in J$. We write $D_{r,J}$ when $I=\{r\}$ is a singleton row index. 

We say \textit{$D_{r,[b,c]}$ is a pipe segment} if the interior part 
$D_{r,[b+1,c-1]}$ consists solely of tiles $\htile$ and $\ptile$.
A \textit{kink} is a pipe segment $D_{r,[b,c]}$ such that $D_{r,b}=\rtile$ and $D_{r,c}=\jtile$.
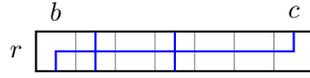
\[
\begin{tikzpicture}[x=1.5em,y=1.5em,thick,color = blue]
\draw[step=1,gray,thin] (0,0) grid (7,1);
\draw[color=black, thick, sharp corners] (0,0) rectangle (7,1);
\draw(0.5,0)--(0.5,0.5)--(6.5,0.5)--(6.5,1.0);
\draw(1.5,0)--(1.5,1);
\draw(3.5,0)--(3.5,1);
\node [color=black] at (-.5,.5) {$r$};
\node [color=black] at (.5,1.5) {$b$};
\node [color=black] at (6.5,1.5) {$c$};
\end{tikzpicture}
\]

By definition, a non-blank tile
$D_{r, [b, b]}$ is always a pipe segment.
We also make the following observation
which 
says gluing two 
pipe segments produces a pipe segment.
\begin{lem}
\label{L: Combine pipe segment}
Suppose $D_{r, [a,b]}$ and $D_{r, [b,c]}$
are pipe segments, 
then so is $D_{r, [a, c]}$.
\end{lem}
\begin{proof}
For $a < j < c$, $D_{r, j}$ is clearly
$\htile$ or $\ptile$ if $j \neq b$.
Since $D_{r, [b,c]}$ is a pipe segment, 
$D_{r, b}$ can be $\rtile$, $\htile$
or $\ptile$.
Since $D_{r, [a,b]}$ is a pipe segment, 
we know $D_{r, b} \neq \rtile$,
so it is $\htile$ or $\ptile$.
\end{proof}

A \textit{light sequence} in $D\in \MBPD(n)$ is a set of tiles of the form $D_{r,[b,c]}$ 
which consists solely of light tiles. The RJ subsequence of a light sequence is the sequence of tiles obtained by only keeping the $\rtile$'s and $\jtile$'s. They come in four flavors.
\begin{enumerate}
    \item Paired: Some number of copies of a pair given by an $\rtile$ followed by a $\jtile$.
    \item Type J: A $\jtile$ followed by a paired sequence.
    \item Type R: A paired sequence followed by an $\rtile$.
    \item Type JR: A $\jtile$ followed by a paired sequence followed by a $\rtile$.
\end{enumerate}

\begin{rem} \label{R:light sequence} Consider a light sequence $S$ in a row of some $D\in\MBPD(n)$.
\begin{enumerate}
    \item \label{I:light after no right} If $S$ immediately follows a $\btile$, $\vtile$, $\jtile$, or $\mtile$ then $S$ is paired or type R.
    \item \label{I:light after right} If $S$ immediately follows a $\htile$, $\ptile$, or $\rtile$ then $S$ is type J or type JR.
    \item \label{I:light before no left} If $S$ immediately precedes a $\btile$, $\vtile$, or $\rtile$ then $S$ is paired or of type $J$.
    \item \label{I:light before left} If $S$ immediately precedes a $\htile$, $\ptile$, $\jtile$, or $\mtile$ then $S$ must be type R or type JR.
\end{enumerate}
\end{rem}
The following remark is a useful observation on MBPDs.
\begin{rem} \label{R:rtile right of heavy}
For any $D \in\MBPD(n)$ and heavy tile in $D$, there is a $\rtile$ in its row to the right and a $\rtile$ in its column below.
\end{rem}

\subsection{Drooping} 
For a row index $r$ and column indices $1\le b<d\le n$,
the $(r,[b,d])$-droop and $(r,[b,d])$-undroop are operations that change an MBPD into another,
only changing tiles in the ``(un)droop rectangle" $[r,r+1] \times [b,d]$.

We say that $D\in\MBPD(n)$ admits the \textit{$(r,[b,d])$-droop} if 
\begin{itemize}
    \item The droop rectangle contains only light tiles except possibly a $\btile$
    at $(r+1,d)$. 
    \item $D_{r,[b,d]}$ is a pipe segment.
    \item $D_{r+1,[b+1,d-1]}$ is a paired light sequence.
    \item $D_{r,b}\ne \htile$ and $D_{r,d}\ne \ptile$.
   
\end{itemize}
\begin{rem} The pipe segment implies that $D_{r,b}$ is connected to the right
and $D_{r,d}$ is connected to the left. The paired light sequence implies that
$D_{r+1,b}$ is not connected to the right and $D_{r+1,d}$ is not connected to the left.
We deduce that if $D$ admits the $(r,[b,d])$-droop then
\begin{itemize}
\item $D_{r,b}$ is $\rtile$ or $\ptile$ (it connects to the right and down).
\item $D_{r,d}$ is $\jtile$ or $\htile$ (it connects to the left and not down).
\item $D_{r+1,b}$ is $\jtile$ or $\vtile$, (it connects up and not to the right).
\item $D_{r+1,d}$ is $\btile$ or $\rtile$ (it does not connect left nor up).
\end{itemize}
\end{rem}
If $D$ admits the $(r,[b,d])$-droop then we may produce $D'\in\MBPD(n)$
as follows; we say that $D\to D'$ is the $(r,[b,d])$-droop.

In the following pictures we only draw the parts of the upper right and lower left corner tiles $D_{r,d}$ and $D_{r+1,b}$ that change during the (un)droop; the other parts remain the same.

The pipe segment $D_{r,[b,d]}$, which connects down to $D_{r+1,b}$,  ``droops" to a pipe segment $D'_{r+1,[b,d]}$ which is connected upwards to $D'_{r,d}$. In columns $x$ of $D$ for $b<x<d$, the vertical pipes ($x$ such that $b<x<d$ and $D_{r,x}=D_{r+1,x}=\ptile$) are unchanged. Each kink in row $r+1$ in $D$ between columns $b$ and $d$, is shifted up into row $r$ in $D'$. A shifted kink is shaded gray in the following picture of a droop.
\[
\begin{tikzpicture}[x=1.5em,y=1.5em,thick,color = blue]
\draw[step=1,gray,thin] (0,0) grid (7,2);
\draw[color=black, thick, sharp corners] (0,0) rectangle (7,2);
\draw [color=gray, fill=gray] (1,0) rectangle (5,1);
\draw(6.5,1.5)--(0.5,1.5)--(.5,.5);
\draw(1.5,0.0)--(1.5,0.5)--(4.5,0.5)--(4.5,2.0);
\draw(2.5,0.0)--(2.5,2.0);
\draw(5.5,0.0)--(5.5,2.0);
\node [color=black] at (-.5,1.5) {$r$};       
\node [color=black] at (-1,.5) {$r+1$};
\node[color=black] at (.5,2.5) {$b$};
\node[color=black] at (6.5,2.5) {$d$};
\node [color=black] at (3.5,-1) {$D$};
\node[color=black] at (7.6,1) {$\rightarrow$};
\end{tikzpicture} 
\begin{tikzpicture}[x=1.5em,y=1.5em,thick,color = blue]
\draw[step=1,gray,thin] (0,0) grid (7,2);
\draw[color=black, thick, sharp corners] (0,0) rectangle (7,2);
\draw[color=gray,fill=gray] (1,1) rectangle (5,2);
\draw(6.5,1.5)--(6.5,0.5)--(.5,.5);
\draw(1.5,0.0)--(1.5,1.5)--(4.5,1.5)--(4.5,2.0);
\draw(2.5,0.0)--(2.5,2.0);
\draw(5.5,0.0)--(5.5,2.0);
\node [color=black] at (3.5,-1) {$D'$};
\end{tikzpicture}
\]
The naming of droops depends on the upper left and lower right corner tiles of the rectangle. The arrows indicates how the upper left and lower right tiles change during the droop. The upper left tile at $(r,b)$ ``loses an $\rtile$'' while the lower right tile at $(r+1,d)$ ``gains a $\jtile$''. 
\begin{align*}
\begin{array}{c||cc} 
D_{r,b} \backslash D_{r+1,d} & \btile\to \jtile & \rtile\to \ptile \\ \hline \hline
\rtile \to \btile& \text{standard} & \text{fuse}  \\
\ptile\to\jtile &\text{split} &\text{split-fuse}
\end{array}
\end{align*}
The above droop is standard. We imagine in $D$ that the pipe in row $r$ is a blue rope which is pinned in the center of the upper right and lower left boxes, and we are pulling the rope taut and holding it at the center of the upper left box. Then we let go.
The rope falls to the floor (middle of row $r+1$). After it falls the rope forms a $\jtile$ in the lower right.

Here is a fuse droop: the $\rtile$ in the lower right corner, when drooped upon, acquires a $\jtile$ which is fused to it and becomes a $\ptile$.
\[
\begin{tikzpicture}[x=1.5em,y=1.5em,thick,color = blue]
\draw[step=1,gray,thin] (0,0) grid (7,2);
\draw[color=black, thick, sharp corners] (0,0) rectangle (7,2);
\draw(6.5,1.5)--(0.5,1.5)--(.5,.5);
\draw(1.5,0.0)--(1.5,0.5)--(4.5,0.5)--(4.5,2.0);
\draw(2.5,0.0)--(2.5,2.0);
\draw(5.5,0.0)--(5.5,2.0);
\draw(7.0,0.5)--(6.5,0.5)--(6.5,0.0);
\node[color=black] at (7.6,1) {$\rightarrow$};
\end{tikzpicture} 
\begin{tikzpicture}[x=1.5em,y=1.5em,thick,color = blue]
\draw[step=1,gray,thin] (0,0) grid (7,2);
\draw[color=black, thick, sharp corners] (0,0) rectangle (7,2);
\draw(6.5,1.5)--(6.5,0.5)--(.5,.5);
\draw(1.5,0.0)--(1.5,1.5)--(4.5,1.5)--(4.5,2.0);
\draw(2.5,0.0)--(2.5,2.0);
\draw(5.5,0.0)--(5.5,2.0);
\draw(7.0,0.5)--(6.5,0.5)--(6.5,0.0);
\end{tikzpicture}
\]
For the split droop, we imagine that we cut the $\ptile$ so it falls into the superposition of a $\jtile$ and an $\rtile$. The droop loses the $\rtile$ and produces a $\jtile$ in the lower right corner.
\[
\begin{tikzpicture}[x=1.5em,y=1.5em,thick,color = blue]
\draw[step=1,gray,thin] (0,0) grid (7,2);
\draw[color=black, thick, sharp corners] (0,0) rectangle (7,2);
\draw(6.5,1.5)--(0.5,1.5)--(.5,.5);
\draw(1.5,0.0)--(1.5,0.5)--(4.5,0.5)--(4.5,2.0);
\draw(2.5,0.0)--(2.5,2.0);
\draw(5.5,0.0)--(5.5,2.0);
\draw(0.0,1.5)--(0.5,1.5)--(0.5,2.0);
\node[color=black] at (7.6,1) {$\rightarrow$};
\end{tikzpicture} 
\begin{tikzpicture}[x=1.5em,y=1.5em,thick,color = blue]
\draw[step=1,gray,thin] (0,0) grid (7,2);
\draw[color=black, thick, sharp corners] (0,0) rectangle (7,2);
\draw(6.5,1.5)--(6.5,0.5)--(.5,.5);
\draw(1.5,0.0)--(1.5,1.5)--(4.5,1.5)--(4.5,2.0);
\draw(2.5,0.0)--(2.5,2.0);
\draw(5.5,0.0)--(5.5,2.0);
\draw(0.0,1.5)--(0.5,1.5)--(0.5,2.0);
\end{tikzpicture}
\]
Finally, a split-fuse droop, which combines the splitting and fusing.
\[
\begin{tikzpicture}[x=1.5em,y=1.5em,thick,color = blue]
\draw[step=1,gray,thin] (0,0) grid (7,2);
\draw[color=black, thick, sharp corners] (0,0) rectangle (7,2);
\draw(6.5,1.5)--(0.5,1.5)--(.5,.5);
\draw(1.5,0.0)--(1.5,0.5)--(4.5,0.5)--(4.5,2.0);
\draw(2.5,0.0)--(2.5,2.0);
\draw(5.5,0.0)--(5.5,2.0);
\draw(0.0,1.5)--(0.5,1.5)--(0.5,2.0);
\draw(7.0,0.5)--(6.5,0.5)--(6.5,0.0);
\node[color=black] at (7.6,1) {$\rightarrow$};
\end{tikzpicture} 
\begin{tikzpicture}[x=1.5em,y=1.5em,thick,color = blue]
\draw[step=1,gray,thin] (0,0) grid (7,2);
\draw[color=black, thick, sharp corners] (0,0) rectangle (7,2);
\draw(6.5,1.5)--(6.5,0.5)--(.5,.5);
\draw(1.5,0.0)--(1.5,1.5)--(4.5,1.5)--(4.5,2.0);
\draw(2.5,0.0)--(2.5,2.0);
\draw(5.5,0.0)--(5.5,2.0);
\draw(0.0,1.5)--(0.5,1.5)--(0.5,2.0);
\draw(7.0,0.5)--(6.5,0.5)--(6.5,0.0);
\end{tikzpicture}
\]
By definition, $D\to D'$ is a \textit{$(r,[b,d])$-undroop} if and only if 
$D'\to D$ is a $(r,[b,d])$-droop. 
Fusing (resp. splitting) undrooping is inverse to splitting (resp. fusing) drooping.

Explicitly, a \textit{$(r,[b,d])$-undroop} $D\to D'$ is defined when 
\begin{itemize}
    \item The undroop rectangle contains only light tiles except possibly a $\btile$ 
    at $(r,b)$.
    \item $D_{r+1,[b,d]}$ is a pipe segment.
    \item $D_{r,[b+1,d-1]}$ is a paired light sequence.
    \item $D_{r+1,d}\ne \htile$, $D_{r+1,b}\ne\ptile$.
\end{itemize}
\begin{rem}
If $D$ admits a $(r,[b,d])$-undroop then
\begin{itemize}
\item $D_{r+1,d}$ is $\jtile$ or $\ptile$ (it connects to the left and up).
\item $D_{r+1,b}$ is $\htile$ or $\rtile$ (it connects to the right and not up).
\item $D_{r,d}$ is $\vtile$ or $\rtile$ (it connects down and not to the left).
\item $D_{r,b}$ is $\btile$ or $\jtile$ (it does not connect to the right nor down).
\end{itemize}
\end{rem}

It is straightforward to verify that the operations of drooping and undrooping, map MBPDs to MBPDs. 

The following lemma is immediate from the definitions.

\begin{lem} Let $D\in\MBPD(n)$. 
\begin{enumerate}
\item If $D$ admits the $(r,[b,d])$-droop producing $D'$, then
$D'$ admits the $(r,[b,d])$-undroop recovering $D$.
\item If $D'$ admits the $(r,[b,d])$-undroop producing $D$, then
$D$ admits the $(r,[b,d])$-droop recovering $D'$.
\end{enumerate}
\end{lem}

\subsection{Doublecrosses}
For $D\in \MBPD(n)$ the subdiagram $D_{[r,r+1],[b,d]}$ is a \textit{doublecross} if 
$D_{r,[b,d]}$ and $D_{r+1,[b,d]}$ are pipe segments, $D_{r,b}=\rtile$, 
and $D_{r+1,d}=\jtile$.
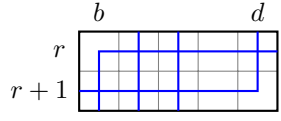
\[
\begin{tikzpicture}[x=1.5em,y=1.5em,thick,color = blue]
\draw[step=1,gray,thin] (0,0) grid (5,2);
\draw[color=black, thick, sharp corners] (0,0) rectangle (5,2);
\draw (.5,0)--(.5,1.5)--(5,1.5);
\draw (0,.5)--(4.5,.5)--(4.5,2);
\draw (1.5,0)--(1.5,2);
\draw (2.5,0)--(2.5,2);
\node[color=black] at (-.5,1.5) {$r$};
\node[color=black] at (-1.0,0.5) {$r+1$};
\node[color=black] at (0.5,2.5) {$b$};
\node[color=black] at (4.5,2.5) {$d$};
\end{tikzpicture}
\]
It is not hard to deduce that 
if $D_{[r, r+1], [b,d]}$
is a double cross, 
then $D_{r+1, b} = D_{r, d} = \ptile$.
This is so named because the pipe through
$(r,b)$ and the pipe through $(r+1, d)$
cross twice at $(r,d)$ and $(r+1,b)$.

\begin{rem}\label{R:doublecrosses unique}
There is at most one doublecross $D_{[r,r+1],[b,d]}$ for each fixed pair $(r,b)$ 
and also at most one such doublecross for a fixed pair $(r,d)$.
\end{rem}

We describe how to find the ``$d$'' given $(r,b)$.
\begin{lem}
\label{L: Find DC}
Suppose $D_{[r, r+1], [b,d]}$
is a double cross.
Then $d$ is the smallest number 
such that $d > b$ and $D_{r+1, d} = \jtile$.
\end{lem}
\begin{proof}
Since $D_{r+1, [b,d]}$ is a pipe segment,
there is no $\jtile$ in $D_{r+1, [b, d)}$.
\end{proof}

\section{$\ftarg$-moves}
\label{sec:f-moves}

In this section and the next, we describe the basic combinatorial operations defined on MBPDs and their inverses  that form the building blocks of our bijection.
\subsection{$f$-targets and $f^*$-targets}
\label{SS:target}
Let $D\in \MBPD(n)$. Say that $(r,c)$ is an \textit{$f$-target} of $D$ if the conditions (f1), (f2), and (f3) hold.
\begin{itemize}
    \item[(f1)] $D_{r,c}$ is the rightmost heavy tile in row $r$.
    \item[(f2)] There is an index $c'>c$ such that $D_{r+1,c'}=\jtile$.
    Let $c'$ be minimum with this property.
    \item[(f3)] All tiles $D_{r+1,j}$ are light for $j<c'$.
\end{itemize}

We say that $(r,c)$ is an \textit{$f^*$-target} of $D$ if the conditions (f1), (f*2), and (f3) hold, where
\begin{itemize}
    \item[(f*2)] There is no $\jtile$ nor $\mtile$ in row $r+1$ to the right of column $c$. In this case let $c'$ be the maximum index such that $D(r,c')=\rtile$. This $\rtile$ exists by Remark \ref{R:rtile right of heavy}.
\end{itemize}
By abuse of language we will also say that $(D,(r,c))$ is an $f$-target (resp. $f^*$-target)
to mean that $D\in\MBPD(n)$ and $(r,c)$ is an $f$-target (resp. $f^*$-target) of $D$.
We will write $F$ to mean either $f$ or $f^*$.
So $(D,(r,c))$ is an $F$-target means it is either an $f$-target or $f^*$-target.

We define the \textit{\window} of an $F$-target $(D,(r,c))$
to be the two-row rectangle $[r,r+1] \times [b,c']$ where 
\begin{align}
    \label{E:f b} 
    b &= \begin{cases}
        c & \text{if $D_{r,c}=\btile$} \\
        \text{maximum $b<c$ with $D_{r,b}=\rtile$} & \text{if $D_{r,c}=\mtile$}
    \end{cases}
\end{align}
and $c'$ is defined by (f2) in the case of an $f$-target and (f*2) in the case of an $f^*$-target.

The \textit{maximum $F$-target} of $D\in\MBPD(n)\setminus \{D_\id\}$ is by definition its bottommost then rightmost heavy tile (say $(r,c)$), which is easily verified to be an $F$-target of $D$. If not explicitly specified otherwise, if we refer to the target of $D$ we mean its maximum $F$-target. We say that an MBPD is \textit{$F$-terminal} if its maximal 
$F$-target is an $f^*$-target and $F$-nonterminal otherwise.

For any $f$-target $(D,(r,c))$ we shall define an $f$-move $D\mapsto f_r(D)$ under which a heavy tile is removed from the $r$-th row and a heavy tile is put into the $(r+1)$-th row. In this case $w(D)=w(f_r(D))$.
If $(D,(r,c))$ is an $f^*$-target we shall define an $f^*$-move
$D\mapsto f_r^*(D)$ in which a heavy tile is removed from the $r$-th row. In this case 
$w(D)=w(f^*_r(D))*s_r$. 
Be aware that if we know $w(D)$,
we cannot determine $w(f^*_r(D))$
which could be either $w(D)$ or $w(D) s_r$ 
(see Proposition~\ref{P: F permutation}).
These moves are defined in \S \ref{SS:do F move} depending on a number of cases which are described below.

\subsection{Two trichotomies of cases for $F$-moves}
There are two trichotomies for an $F$-target $(D,(r,c))$, 
giving nine cases in all. They are called left and right because they describe the left and right sides of the window respectively.

The left trichotomy for the $F$-target $(D,(r,c))$ asserts
that exactly one of $\Blank$, $\Cross$, or $\NCross$ holds.
\begin{itemize}
    \item[($\Blank$)] (Blank): This holds if $D_{r,c}=\btile$.
    \item When not in Case $\Blank$, $D_{r,c}=\mtile$. Let $b$ be as in \eqref{E:f b}.
    \begin{itemize}
        \item[($\Cross$)] (Crossing): 
        $D_{r+1,[b,c]}$ is a pipe segment.
        This is so named because in this case the pipe through $(r,c)$ goes to the left and then crosses vertically with the above pipe segment at $D_{r+1,b}=\ptile$.
        \item[($\NCross$)] (Noncrossing): Neither Case $\Blank$ nor Case $\Cross$ holds.
    \end{itemize}
\end{itemize}
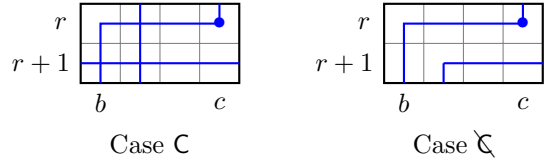
\[
\begin{matrix} \\
\begin{tikzpicture}[x=1.5em,y=1.5em,thick,color = blue]
\draw[step=1,gray,thin] (0,0) grid (4,2);
\draw[color=black, thick, sharp corners] (0,0) rectangle (4,2);
\draw(3.5,2.0)--(3.5,1.5)--(.5,1.5)--(.5,0);
\draw(4.0,.5)--(0,.5);
\draw(1.5,2.0)--(1.5,0);
\node at (3.5,1.5) {$\bullet$};
\node[color=black] at (.5,-.5) {$b$};
\node[color=black] at (3.5,-.5) {$c$};
\node[color=black] at (-.5,1.5) {$r$};
\node[color=black] at (-1,.5) {$r+1$};
\end{tikzpicture} \\
\text{\hphantom{xxxx}Case $\Cross$}
\end{matrix} \\[2mm]
\qquad
\begin{matrix} \\
\begin{tikzpicture}[x=1.5em,y=1.5em,thick,color = blue]
\draw[step=1,gray,thin] (0,0) grid (4,2);
\draw[color=black, thick, sharp corners] (0,0) rectangle (4,2);
\draw(3.5,2.0)--(3.5,1.5)--(.5,1.5)--(.5,0);
\draw(4.0,.5)--(1.5,.5);
\draw(1.5,.5)--(1.5,0);
\node at (3.5,1.5) {$\bullet$};
\node[color=black] at (.5,-.5) {$b$};
\node[color=black] at (3.5,-.5) {$c$};
\node[color=black] at (-.5,1.5) {$r$};
\node[color=black] at (-1,.5) {$r+1$};
\end{tikzpicture} \\
\text{\hphantom{xxxx}Case $\NCross$}
\end{matrix} \\[2mm]
\]

The right  trichotomy says that for an $F$-target $(D,(r,c))$, 
exactly one $\Terminal$, $\DCross$, or $\Ordinary$ holds.
\begin{itemize}
    \item [($\Terminal$)] (Terminal) The tile $(r,c)$ is an $f^*$-target of $D$.
\begin{align*}
&\begin{tikzpicture}[x=1.5em,y=1.5em,thick,color = blue]
\draw[step=1,gray,thin] (0,0) grid (6,2);
\draw[color=black, thick, sharp corners] (0,0) rectangle (6,2);
\draw(6,1.5)--(3.5,1.5)--(3.5,0);
\draw(6,.5)--(0,.5);
\draw(4.5,2)--(4.5,0);
\draw(2.5,2)--(2.5,0);
\draw(1.5,2)--(1.5,0);
\node[color=black] at (.5,-.5) {$c$};
\node[color=black] at (3.5,-.5) {$c'$};
\node[color=black] at (5.5,-.5) {$n$};
\node[color=black] at (6.5,1.5) {$r$};
\node[color=black] at (7,.5) {$\,r+1$};
\draw[color=red, thick, sharp corners] (0,1) rectangle (1,2);
\end{tikzpicture}
\\
&\hphantom{xxxxxx}\text{Case $\Terminal$}
\end{align*}
    \item The complement of  Case $\Terminal$ : $(r,c)$ is an $f$-target of $D$. 
    There are two subcases.
    \begin{itemize}
        \item[($\DCross$)] (Doublecross): There is a doublecross $D_{[r,r+1], [d,c']}$ for some $d$ with $c<d<c'$.
        \item[($\Ordinary$)] (Ordinary): Cases $\Terminal$ and $\DCross$ do not hold.
    \end{itemize}
\end{itemize}
The doublecross requires a $\jtile$ in row $r+1$ not allowed by Case $\Terminal$. Such a doublecross is unique by Remark \ref{R:doublecrosses unique}.

Given an $\ftarg$ move, define the right droop column index $\rho$ by
\begin{align}\label{E:F rho}
\rho = \begin{cases}
    c' & \text{for Cases $\Terminal$ or $\Ordinary$} \\
    d & \text{for Case $\DCross$.}
\end{cases}
\end{align}
To be parallel with the upcoming definitions for $E$-moves, define
the left droop column index $\lambda$ by
\begin{align}\label{E:F lambda}
\lambda = c
\end{align}
which is the column index of the $F$-target.

The following Lemma is used to prove the well-definedness of the $\ftarg$-moves to be defined later.

\begin{lem} \label{L:TDO}
Suppose $(D,(r,c))$ is an $F$-target.
\begin{enumerate}
\item \label{I:bottom left} $D_{r+1,c}$ is either $\htile$ or $\rtile$.
\item \label{I:pipe} $D_{r+1,[c,c']}$ is a pipe segment.
\item \label{I:bottom right} In Case $\Terminal$, $D_{r+1,c'}=\ptile$. Moreover there are no $\rtile$s in row $r+1$ to the right of column $c$.
\item \label{I:light} $S=D_{r,[c+1,c'-1]}$ is a light sequence that is either paired or of type $R$.
\begin{enumerate}
  \item \label{I:typeR} $S$ is of type $R$ if and only if $D_{r,c'}=\ptile$. In this case,
    suppose $S$ has last $\rtile$ at $D_{r,d}$ and we are not in Case $\Terminal$.
    Then $D_{[r,r+1],[d,c']}$ is a doublecross (Case $\DCross$ holds).
    \item \label{I:paired} $S$ is paired if and only if $D_{r,c'}$ is $\vtile$ or $\rtile$. Then either Case $\Terminal$  or Case $\Ordinary$ holds.
\end{enumerate}
\end{enumerate}
Moreover the $(r,[c,\rho])$-droop applies (after removing the marking if $D_{r,c}=\mtile$).
\end{lem}
\begin{proof} By (f3) $D_{r+1,[c,c'-1]}$ is a light sequence.
In particular $D_{r+1,c}$ is light.
Since $D_{r,c}$ is heavy, $D_{r+1,c}$ does not connect upwards.
Item \eqref{I:bottom left} follows.

By Item \eqref{I:bottom left} $D_{r+1,c}$ connects to the right. Following this pipe to the right in row $r+1$, suppose it turns before column $c'$. It turns at $\jtile$ or $\mtile$. This is a contradiction for the $f$-target case by (f2) and (f3), and for the $f^*$-target case by (f*2).
Item \eqref{I:pipe} follows. Moreover $D_{r+1,c'}$ must connect to the left.
In the case that $(r,c)$ is an $f^*$-target of $D$, $D_{r+1,c'}$ must also connect upwards since $D_{r,c'}=\rtile$. Therefore $D_{r+1,c}$ is $\rtile$ or $\htile$. 
The pipe leaving $(r+1,c)$ to the right, cannot turn upwards.
Item \eqref{I:bottom right} follows.

For \eqref{I:light}, by (f1) $S$ is a light sequence. 
It is preceded by a heavy tile $D_{r,c}$, so $S$ must be paired or of type R.
$S$ is followed by the tile $D_{r,c'}$, which is light by (f1)
and which connects downwards, since $D_{r+1,c'}$ is either $\jtile$ or $\ptile$. So $D_{r,c'}$ must be $\vtile$, $\rtile$, or $\ptile$.

If $S$ is of paired type then $D_{r,c'}$ does not connect to the left,
and is therefore $\vtile$ or $\rtile$.  If $S$ is of type $R$ then
$D_{r,c'}$ connects to the left and is therefore $\ptile$.

Suppose $S$ is of paired type. By Item \eqref{I:bottom left} $D_{r+1,c}$ does not connect upwards and in our current case $D_{r,c'}$ does not connect to the left. It follows that 
$D$ admits the $(r,[c,c'])$-undroop, proving \eqref{I:paired}.

Otherwise $S$ is of type R and $D_{r,c'}=\ptile$. 
Let $S$ have last $\rtile$ at position $(r,d)$ where $c<d<c'$.
Suppose also that we are in Case $\Terminal$ so that $D_{r+1,c'}=\jtile$ by
Item \eqref{I:bottom right}.
Now $D_{r+1,d}$ connects upwards since $D_{r,d}=\rtile$ connects downwards
and $D_{r+1,d}$ connects to the right by Item \eqref{I:pipe}.
Therefore $D_{r+1,d}=\ptile$. By the definition of $d$
$D_{r,[d,c']}$ is a pipe segment. By Item \eqref{I:pipe}
$D_{r+1,[d,c']}$ is a pipe segment. We have verified that $D_{[r,r+1]\times [d,c']}$ is a doublecross. 

We have $D_{r,d}= \rtile$ and $D_{r+1,d}=\ptile$.
$D_{r+1,c}$ does not connect upwards since $D_{r,c}$ is heavy.
$D_{r+1,[c,d]}$ is a pipe segment by Item \eqref{I:pipe}
and $D_{r,[c,d]}$ is a paired light sequence since the ending
$\rtile$ of the type $R$ light sequence $S$ has been removed.
Hence the $(r,[c,d])$-undroop of $D$ is well-defined and \eqref{I:typeR} holds.
\end{proof}

\subsection{$F$-moves}
\label{SS:do F move}
Let $(D,(r,c))$ be an $F$-target. 
We define $F_r(D)\in \MBPD(n)$ to be the result
of the following steps:
\begin{itemize}
\item If $D_{r,c}=\mtile$ remove the marking.
\item In Cases $\Blank$ and $\Cross$ perform the $(r,[c,\rho])$-undroop; see \eqref{E:F rho}.
\item If the tile at $(r+1,c')$ is $\jtile$, mark it.
\end{itemize}

The following proposition summarizes how
 $F$-moves affect the permutation associated to the MBPD.
\begin{prop}
\label{P: F permutation}
   
    Let $(D,(r,c))$ be an $F$-target. Then in all the non-terminal cases, $w(f_r(D))=w(D)$. In the terminal cases,
    $w(f^*_r(D))*s_r=w(D)$.

\end{prop}
\begin{proof}
    In Case $\Terminal$, by the definition of $\fs_r$-target, the pipes $p$ and $q$ that enter from row $r$ and $r+1$ necessarily cross within the window of the target. Therefore, $s_r$ is a right descent of $w(D)$.  In Case $\NCross$, only the marking 
    changes, so the permutation does not
    change. When not in Cases $\NCross$ or
    $\Terminal$, the $(r,[c,\rho])$-undroop first possibly remove a crossing from a doublecross, then possibly create a doublecross, so the permutation is unchanged. Finally in Cases $\Terminal\Blank$ and $\Terminal\Cross$, the unique crossing between pipes $p$ and $q$ within the window is uncrossed, and a new double cross which does not affect the permutation is created in Case $\Terminal\Cross$. 
    If the pipes that enter from row $r$ and $r+1$ in $\fs_r(D)$ no longer cross then
    $w(\fs_r(D))=w(D)s_r$; otherwise $w(\fs_r(D))=w(D)$.
    The claim then follows.
 \end{proof}

The $F_r$ move 
removes a heavy tile in row $r$ 
and creates a heavy tile in row $r+1$
unless it is in Case $\Terminal$.
This behavior is captured by the following lemma.
\begin{lem}
\label{L: Only heavy after F}
Assume $(D,(r,c))$ is an $F$-target
with window $[r, r+1] \times [b, c']$.
Then all tiles in $F_r(D)_{[r, r+1], [b, c']}$
are light except possibly $F_r(D)_{r+1, c'}$.
It is light if and only if $(D,(r,c))$
is in Case $\Terminal$.
\end{lem}
\begin{proof}
By (f1) and (f3),
the only heavy tile in $D_{[r, r+1], [b, c']}$ is $D_{r, c}$.
We first prove the lemma in Case $\NCross$.
Since $D_{r, c} = \mtile$,
the $F_r$ move unmarks it.
Then $F_r$ marks $D_{r+1, c'}$ if and only
if $D_{r+1, c'} = \jtile$,
which is equivalent to 
we are not in case $\Terminal$.

Next we prove the lemma
outside of Case $\NCross$.
After the $(r,[c,\rho])$-undroop,
the only tile that might be heavy 
in $[r, r+1] \times [b, c']$
is $(r+1, \rho)$.
In Case $\Terminal$,
$D_{r+1, \rho} = D_{r+1, c'} = \ptile$,
so $F_r(D)_{r+1, \rho} = \rtile$ is not heavy. 
In Case $\DCross$,
$D_{r+1, \rho} = \ptile$,
so $F_r(D)_{r+1, \rho} = \rtile$ is not heavy. 
Then $D_{r, c'} = \jtile$ is marked
and become the only heavy tile
in $[r, r+1] \times [b, c']$.
In Case $\Ordinary$,
$D_{r+1, \rho} = D_{r+1, c'} = \jtile$,
so $F_r(D)_{r+1, \rho} = \btile$ is heavy.
\end{proof}

We now give illustrations for the $F$-moves in all nine cases. The labels for the results of these moves will be explained later in Section~\ref{sec:e-moves} when we define $E$-moves.
The window of the move is indicated by a thick two-row rectangle. 
The box $(r,\lambda)$ is colored green and the box $(r+1,\rho)$ is colored red. When (un)droops occur, these boxes will be the corners of the (un)droop window.
\[

\]

\section{$\etarg$-moves}
\label{sec:e-moves}
\subsection{$\etarg$-targets}
We say the tile $(r+1, c)$ is an \textit{$e$-target} of $D\in \MBPD(n)$ if the following are satisfied. 
\begin{itemize}
\item[(e1)] It is the leftmost heavy tile on row $r+1$.
\item[(e2)] There is an index $c' < c$ such that $D_{r,c'}=\rtile$
and $D_{r+1,[c',c]}$ is not a pipe segment.
Let $c'$ be maximum with this property.
\item[(e3)] On the right of $(r, c')$ in row $r$ there are no heavy tiles. 
\end{itemize}
We say the tile $(r+1,c)$ is an \textit{$e^*$-target} of $D\in \MBPD(n)$ if the 
conditions (e*1), (e2), and (e3) are satisfied, where
the condition (e1) for an $e$-target has been replaced by the condition
\begin{itemize}
\item[(e*1)] There are no heavy tiles in row $r+1$,
and $c$ is the largest such that $D_{r,c}$ or $D_{r+1, c}$ is $\rtile$.
\end{itemize}

Again we abuse language by saying that an 
\emph{$e$-target} (resp. $e^*$-target) is a pair $(D,(r+1,c))$ where $D\in\MBPD(n)$
and $(r+1,c)$ is an $e$-target (resp. $e^*$-target) of $D$. Similar to the $F$-case, an $E$-target is either an $e$-target or $e^*$-target.

We define the \emph{\window } of an $E$-target $(D,(r+1,c))$ to be the two-row rectangle
$[r, r+1] \times [c', c]$
with $c'$ as in (e2). 

If $(D, (r+1,c))$ is an $e$-target, we shall define an $e$-move $D\mapsto e_r(D)$ under which a heavy tile is removed from the $(r+1)$-th row and a heavy tile is put into the $r$-th row. In this case $w(D)=w(e_r(D))$. 
If $(D, (r+1,c))$ is an $e^*$-target we shall define an $e^*$-move $D\mapsto e^*_r(D)$ in which a heavy tile is inserted into the $r$-th row. In this case $w(e^*_r(D))=w(D)*s_r$. 
The moves will only affect tiles in the windows.

After some preliminary work these moves are defined in \S \ref{SS:the E moves}.

\subsection{Two trichotomies of cases for $\etarg$-moves}
Let $(D,(r+1,c))$ be an $E$-target.
We have 9 cases with three choices for the first symbol and three for 
the second symbol.

Here is the right  trichotomy for $E$-moves.
Exactly one of them holds for an $E$-target.
\begin{itemize}
\item $\Initial$ (Initial): $(r+1, c)$ is an $e^*$-target.
\item $\Plus$ (Plus): $(r+1, c)$ is an $e$-target and $D_{r,c}=\ptile$. 

\item $\NPlus$ (No plus): $(r+1, c)$ is an $e$-target and $D_{r, c}\ne\ptile$. 
\end{itemize}

Define the \emph{right droop column} $\rho$ of the $E$-target $(D,(r+1,c))$ by
\begin{align}\label{E:right droop column}
  \rho = \begin{cases} 
  c & \text{in Cases $\Initial$ and $\NPlus$} \\
  \max\bigg\{d' \mid \text{$d'<c$ and $D_{r+1,d'}=\rtile$}\bigg\} & \text{in Case $\Plus$.} 
  \end{cases}
\end{align}

The left  trichotomy for $\etarg$-moves
are the following three cases, exactly one of which holds for any $E$-target $(D,(r+1,c))$
with window $[r,r+1]\times [c',c]$.

\begin{itemize}
\item $\Leftturn$ (Left turn): $D_{r,[c',c]}$ is not a pipe segment. That is, the pipe leaving the $\rtile$ at $(r,c')$ to the right, must turn upwards (make a left turn) before arriving at $(r,c)$.
\item $\DCross$ (Doublecross): $D_{r,[c',c]}$ is a pipe segment and 
the tile $(r,c')$ is the top-left corner of a doublecross.
\item $\Straight$ (Straight): $D_{r,[c',c]}$ is a pipe segment
and the tile $(r,c')$ is not the top-left corner of a doublecross.
\end{itemize}

Define the \emph{left droop column} $\lambda$ of the $E$-target $(D,(r+1,c))$ by
\begin{align}\label{E:left droop column}
  \lambda = \begin{cases} 
  c' & \text{in Case $\Straight$} \\
  \min \bigg\{\text{$d\mid d>c'$ such that $D_{r+1,d}=\jtile$}\bigg\} & \text{in Case $\DCross$}\\
  \min \bigg\{\text{$d\mid d>c'$ such that $D_{r,d}=\jtile$}\bigg\} & \text{in Case $\Leftturn$.}
  \end{cases}
\end{align}

We start by characterizing what the tile
$(r+1, c)$ can be if it is an $E$-target.
\begin{lem}
\label{L: What can E target be}
Let $D$ have $\etarg$-target $(r+1,c)$.
In Case $\Leftturn$,
$D_{r+1, c}$ is $\ptile$ or $\mtile$.
Otherwise, 
$D_{r+1, c}$ is $\rtile$, $\btile$ or $\mtile$. 
\end{lem}

\begin{proof}
We first show that $D_{r+1,c}$ is one of $\rtile$, $\btile$, $\ptile$, or $\mtile$.
Then we show that if $D_{r+1,c}$ is $\rtile$ or $\btile$ then $(D,(r+1,c))$ is not in case $\Leftturn$,
and if $D_{r+1,c}=\ptile$ then $(D,(r+1,c))$ is in case $\Leftturn$.

In Case $\Initial$,
we know $D_{r, c}$ or $D_{r+1, c}$
is the only $\rtile$ in $D_{[r, r+1], [c, n]}$.
In the former situation, 
we know the rightmost $\rtile$ in row $r+1$
is on the left of $D_{r+1, c}$,
so $D_{r+1, c} = \ptile$.
In the latter situation, 
$D_{r+1, c} = \rtile$.
If we are not in Case $\Initial$,
$D_{r+1, c}$ is heavy, 
so it is $\btile$ or $\mtile$.

Suppose 
$D_{r+1, c}$ is $\rtile$ or $\btile$,
so $D_{r, c}$ is not connected to the bottom.
By (e3), $D_{r, c}$ is light, 
so it connects to the left. 
Let $d < c$ be the largest such that 
$D_{r, d} = \rtile$.
Then $D_{r+1, [d, c]}$ is not a pipe 
segment. 
We must have $d = c'$ and $D_{r,[c', c]}$
is a pipe segment,
so we cannot be in Case $\Leftturn$.

Suppose $D_{r+1, c} = \ptile$.
We are in Case $\Initial$ since it is light.
By (e*1), $D_{r, c} = \rtile$.
Then $D_{r, [c',c]}$ cannot be a pipe segment,
so we are in Case $\Leftturn$.
\end{proof}

\begin{lem} \label{L:e SDL}
Let $D$ have $\etarg$-target $(r+1,c)$.
\begin{enumerate}
\item \label{I: e L} 
If Case $\Leftturn$ holds then $S=D_{r,[c'+1,c-1]}$ has type J or JR.
\item \label{I: e top pipe} Suppose Case $\Leftturn$ does not hold.
Let $S'=D_{r+1,[c'+1,c-1]}$.
Case $\DCross$ holds if and only if 
$S'$ is of type J or JR.
In this case, 
$D_{[r,r+1],[c',\lambda]}$ is a doublecross.
\end{enumerate}
\end{lem}
\begin{proof} Suppose Case $\Leftturn$ holds. 
Since $D_{r,c'}=\rtile$ we deduce Item \eqref{I: e L} by 
Remark \ref{R:light sequence} \eqref{I:light after right}.

Suppose Case $\Leftturn$ does not hold,
so $D_{r,[c',c]}$ is a pipe segment.
If $S'$ has type J or JR,
we can find smallest $d$ such that 
$c' < d < c$ and $D_{r+1, d} = \jtile$.
Then $D_{r+1, [c', d]}$ is a pipe segment.
Since $D_{r,[c', c]}$ is a pipe segment,
so is $D_{r, [c', d]}$.
We have $D_{[r, r+1], [c, d]}$ is a double cross,
so we are in Case $\DCross$ and $\lambda = d$.
Conversely, if the doublecross exists,
then $D_{r+1, [c, \lambda]}$
is a pipe segment with $D_{r+1, \lambda} = \jtile$.
Thus, $S'$ has type J or JR.
\end{proof}

\begin{lem}
If not in Case $\Leftturn$,
$D$ admits the $(r, [\lambda, \rho])$-droop.
\end{lem}
\begin{proof}
We check the conditions of drooping.
\begin{itemize}
\item All tiles in $D_{[r, r+1], [c', c]}$
are not heavy except possibly $(r+1, c)$.
Since $c' \leq \lambda < \rho \leq c$,
all tiles in $D_{[r, r+1], [\lambda, \rho]}$
are not heavy except possibly $(r+1, \rho)$.
\item Since we are not in Case $\Leftturn$,
$D_{r, [c', c]}$ is a pipe segment and thus
 $D_{r, [\lambda, \rho]}$ is also.
\item Let $S' = D_{r+1, [\lambda+1, \rho - 1]}$.
To show $S'$ is paired, 
we first show $S'$ cannot have type J or JR.
\begin{itemize}
\item In case $\Straight$, we know $\lambda = c'$.
By Lemma~\ref{L:e SDL},
$D_{r+1,[c'+1,c-1]}$ does not have type J or JR.
Thus, $S'$ does not have type J or JR.
\item In case $\DCross$, we know $D_{r+1, \lambda} = \jtile$.
Thus, $S'$ does not have type J or JR.
\end{itemize}
Finally, we just need to check $S'$ cannot have type R,
which is equivalent to showing that $D_{r+1, \rho}$
does not connect to the left.
Suppose it does.
Then $D_{r+1, \rho}$
cannot be $\rtile$,
so we are not in Case $\Plus$
and $\rho = c$.
By Lemma~\ref{L: What can E target be}, since $D_{r+1,c}$ connects to the left, it is $\ptile$ or $\mtile$,
so it connects to the top.
Since we are not in Case $\Plus$,
$D_{r, c}$ must be $\rtile$ or $\vtile$.
Thus, $D_{r, [c', c]}$ cannot be a pipe segment,
which is a contradiction. 
\item We check $D_{r, \lambda} \neq \htile$.
It is $\rtile$ in Case $\Straight$
and $\jtile$ in Case $\Leftturn$.
In case $\DCross$,
we know $D_{r+1, \lambda} = \jtile$,
so $D_{r, \lambda}$ cannot be $\htile$.

Then we check $D_{r, \rho} \neq \ptile$.
In Case $\Plus$, 
$D_{r+1, \rho} = \rtile$ so $D_{r, \rho} \neq \ptile$.
Otherwise, $\rho = c$ and we know 
$D_{r, c} \neq \ptile$. \qedhere
\end{itemize}
\end{proof}

\subsection{The $\etarg$-moves}
\label{SS:the E moves}
We shall now define $e$-moves, which are the inverses of $f$-moves.
Given an $E$-target $(D,(r+1,c))$, define $E_r(D)\in\MBPD(n)$ to be the result of the
following steps:
\begin{itemize}
\item If $D_{r+1,c}=\mtile$, remove the marking.
\item In Cases $\Straight$ and $\DCross$, perform the
$(r,[\lambda,\rho])$-droop where $\lambda$ and $\rho$ are defined respectively by
\eqref{E:left droop column} and \eqref{E:right droop column}.
\item If $(r, \lambda)$ is $\jtile$, mark it.
\end{itemize}

The $E_r$ move 
creates a heavy tile in row $r$ 
and removes a heavy tile in row $r+1$
unless it is in Case $\Initial$.
This behavior is captured by the following lemma.
\begin{lem}
\label{L: Only heavy after E}
Assume $D_{r+1,c}$ is a $E$-target
with window $[r, r+1] \times [c', c]$.
Then all tiles in $E_r(D)_{[r, r+1], [c', c]}$
are light except $E_r(D)_{r, \lambda}$,
where $\lambda$ is the left droop column.
Moreover, $E_r(D)_{r+1, c} = \ptile$
if we are in Case $\Initial$ and $\jtile$ 
otherwise.
\end{lem}
\begin{proof}
By (e1), (e*1) and (e3),
Tiles in $D_{[r, r+1], [c', c]}$
are all light except $D_{r+1, c}$,
which is light if and only if
we are in Case $\Initial$.
We first prove the lemma in Case
$\Leftturn$.
By Lemma~\ref{L: What can E target be},
$D_{r+1, c} = \ptile$ in Case $\Initial$
and $\mtile$ otherwise. 
Thus, $E_r(D)_{r+1, c} = \ptile$ in Case $\Initial$
and $\jtile$ otherwise.
In either case, it is light.
Since $D_{r, \lambda} = \jtile$,
if will be marked and become
the only heavy tile in $E_r(D)_{[r, r+1], [c', c]}$

Next we prove the lemma
outside of Case $\Leftturn$.
We first understand $E_r(D)_{r, \lambda}$
by the three possibilities of $D_{r+1, c}$
in Lemma~\ref{L: What can E target be}.
\begin{itemize}
\item If $D_{r, c} = \rtile$,
we know we are in Case $\Initial$ since it is light
and Case $\NPlus$ since it is not connected 
to top.
Then $\rho = c$.
After the droop, 
we know $E_r(D)_{r, c} = \ptile$.
\item If $D_{r, c} = \btile$,
we know we are not in Case $\NPlus$
since it is light and not connected to top.
Then $\rho = c$, 
and if would become $\jtile$
after the droop.
\item If $D_{r, c} = \mtile$,
then $E_r$ would unmark it. 
It will stay $\jtile$ after the droop.
\end{itemize}
After the $(r, [\lambda, \rho])$-droop,
$(r, \lambda)$ is 
the only tile that might be heavy
in $[r, r+1] \times [\lambda, \rho]$,
and therefore in $[r, r+1] \times [c', c]$
We know $(r, \lambda)$ is $\btile$ 
or $\jtile$ after the droop.
In the latter case, $E_r$ would mark it,
so $E_r(D)_{r, \lambda}$ must be heavy.
\end{proof}

\section{Tables of cases}
\newcommand{\boxit}[1]{\parbox{.85in}{#1}}

We give a diagrammatic summary of the $F$- and $E$-moves, as described in Section \ref{sec:f-moves} and \ref{sec:e-moves}, for the readers' convenience. Small examples of \window{s} of the nine cases for $f$ and $f^*$ are pictured 
in Figure \ref{F:f 9 cases}. 
The green square has coordinate $(r,\lambda)$
and the red square has coordinate $(r+1,\rho)$. In $F$-move cases other than $\NCross$
and $E$-move cases other than  $\Leftturn$, the green and red squares are the corners of the 
(un)droop window. The parts of tiles not pictured
remain the same. Small examples of the nine cases for $e$ and $e^*$ are pictured 
in Figure \ref{F:e 9 cases}.
\newcommand{\BT}{
\begin{tikzpicture}[x=1.5em,y=1.5em,thick,color = blue]
\draw[step=1,gray,thin] (0,0) grid (4,2);
\draw[color=black, thick, sharp corners] (0,0) rectangle (4,2);
\draw(1.5,2.0)--(1.5,1.0);
\draw(2.5,2.0)--(2.5,1.0);
\draw(4.0,1.5)--(3.5,1.5)--(3.5,1.0);
\draw(1.0,0.5)--(0.5,0.5);
\draw(1.5,1.0)--(1.5,0.0);
\draw(2.0,0.5)--(1.0,0.5);
\draw(2.5,1.0)--(2.5,0.0);
\draw(3.0,0.5)--(2.0,0.5);
\draw(3.5,1.0)--(3.5,0.0);
\draw(4.0,0.5)--(3.0,0.5);
\draw[step=1,red,thick] (3,0) grid (4,1);
\draw[step=1,green,thick] (0,1) grid (1,2);
\end{tikzpicture}
}
\newcommand{\BD}{
\begin{tikzpicture}[x=1.5em,y=1.5em,thick,color = blue]
\draw[step=1,gray,thin] (0,0) grid (4,2);
\draw[color=black, thick, sharp corners] (0,0) rectangle (4,2);
\draw(1.5,2.0)--(1.5,1.0);
\draw(3.0,1.5)--(2.5,1.5)--(2.5,1.0);
\draw(3.5,2.0)--(3.5,1.0);
\draw(4.0,1.5)--(3.0,1.5);
\draw(1.0,0.5)--(0.5,0.5);
\draw(1.5,1.0)--(1.5,0.0);
\draw(2.0,0.5)--(1.0,0.5);
\draw(2.5,1.0)--(2.5,0.0);
\draw(3.0,0.5)--(2.0,0.5);
\draw(3.5,1.0)--(3.5,0.5)--(3.0,0.5);
\draw[step=1,red,thick] (2,0) grid (3,1);
\draw[step=1,green,thick] (0,1) grid (1,2);
\end{tikzpicture}
}
\newcommand{\BO}{
\begin{tikzpicture}[x=1.5em,y=1.5em,thick,color = blue]
\draw[step=1,gray,thin] (0,0) grid (4,2);
\draw[color=black, thick, sharp corners] (0,0) rectangle (4,2);
\draw(1.5,2.0)--(1.5,1.0);
\draw(2.5,2.0)--(2.5,1.0);
\draw(3.5,1.5)--(3.5,1.0);
\draw(1.0,0.5)--(0.5,0.5);
\draw(1.5,1.0)--(1.5,0.0);
\draw(2.0,0.5)--(1.0,0.5);
\draw(2.5,1.0)--(2.5,0.0);
\draw(3.0,0.5)--(2.0,0.5);
\draw(3.5,1.0)--(3.5,0.5)--(3.0,0.5);
\draw[step=1,red,thick] (3,0) grid (4,1);
\draw[step=1,green,thick] (0,1) grid (1,2);
\node[color=black] at (0.5,-.508) {$c$};
\node[color=black] at (3.5,-.4) {$c'$};
\end{tikzpicture}
}
\newcommand{\CT}{
\begin{tikzpicture}[x=1.5em,y=1.5em,thick,color = blue]
\draw[step=1,gray,thin] (0,0) grid (4,2);
\draw[color=black, thick, sharp corners] (0,0) rectangle (4,2);
\draw(1.0,1.5)--(0.5,1.5)--(0.5,1.0);
\draw(1.5,2.0)--(1.5,1.5)--(1.0,1.5);
\node at (1.5,1.5) {$\bullet$};
\draw(2.5,2.0)--(2.5,1.0);
\draw(4.0,1.5)--(3.5,1.5)--(3.5,1.0);
\draw(0.5,1.0)--(0.5,0.0);
\draw(1.0,0.5)--(0.0,0.5);
\draw(2.0,0.5)--(1.0,0.5);
\draw(2.5,1.0)--(2.5,0.0);
\draw(3.0,0.5)--(2.0,0.5);
\draw(3.5,1.0)--(3.5,0.0);
\draw(4.0,0.5)--(3.0,0.5);
\draw[step=1,red,thick] (3,0) grid (4,1);
\draw[step=1,green,thick] (1,1) grid (2,2);
\end{tikzpicture}
}

\newcommand{\CD}{
\begin{tikzpicture}[x=1.5em,y=1.5em,thick,color = blue]
\draw[step=1,gray,thin] (0,0) grid (4,2);
\draw[color=black, thick, sharp corners] (0,0) rectangle (4,2);
\draw(1.0,1.5)--(0.5,1.5)--(0.5,1.0);
\draw(1.5,2.0)--(1.5,1.5)--(1.0,1.5);
\node at (1.5,1.5) {$\bullet$};
\draw(3.0,1.5)--(2.5,1.5)--(2.5,1.0);
\draw(3.5,2.0)--(3.5,1.0);
\draw(4.0,1.5)--(3.0,1.5);
\draw(0.5,1.0)--(0.5,0.0);
\draw(1.0,0.5)--(0.0,0.5);
\draw(2.0,0.5)--(1.0,0.5);
\draw(2.5,1.0)--(2.5,0.0);
\draw(3.0,0.5)--(2.0,0.5);
\draw(3.5,1.0)--(3.5,0.5)--(3.0,0.5);
\draw[step=1,red,thick] (2,0) grid (3,1);
\draw[step=1,green,thick] (1,1) grid (2,2);
\end{tikzpicture}}

\newcommand{\CO}{
\begin{tikzpicture}[x=1.5em,y=1.5em,thick,color = blue]
\draw[step=1,gray,thin] (0,0) grid (4,2);
\draw[color=black, thick, sharp corners] (0,0) rectangle (4,2);
\draw(1.0,1.5)--(0.5,1.5)--(0.5,1.0);
\draw(1.5,2.0)--(1.5,1.5)--(1.0,1.5);
\node at (1.5,1.5) {$\bullet$};
\draw(2.5,2.0)--(2.5,1.0);
\draw(3.5,1.5)--(3.5,1.0);
\draw(0.5,1.0)--(0.5,0.0);
\draw(1.0,0.5)--(0.0,0.5);
\draw(2.0,0.5)--(1.0,0.5);
\draw(2.5,1.0)--(2.5,0.0);
\draw(3.0,0.5)--(2.0,0.5);
\draw(3.5,1.0)--(3.5,0.5)--(3.0,0.5);
\draw[step=1,red,thick] (3,0) grid (4,1);
\draw[step=1,green,thick] (1,1) grid (2,2);
\node[color=black] at (1.5,-.508) {$c$};
\node[color=black] at (3.5,-.4) {$c'$};
\end{tikzpicture}
}

\newcommand{\NT}{
\begin{tikzpicture}[x=1.5em,y=1.5em,thick,color = blue]
\draw[step=1,gray,thin] (0,0) grid (4,2);
\draw[color=black, thick, sharp corners] (0,0) rectangle (4,2);
\draw(1.0,1.5)--(0.5,1.5)--(0.5,1.0);
\draw(1.5,2.0)--(1.5,1.5)--(1.0,1.5);
\node at (1.5,1.5) {$\bullet$};
\draw(2.5,2.0)--(2.5,1.0);
\draw(4.0,1.5)--(3.5,1.5)--(3.5,1.0);
\draw(0.5,1.0)--(0.5,0.5);
\draw(2.0,0.5)--(1.5,0.5)--(1.5,0.0);
\draw(2.5,1.0)--(2.5,0.0);
\draw(3.0,0.5)--(2.0,0.5);
\draw(3.5,1.0)--(3.5,0.0);
\draw(4.0,0.5)--(3.0,0.5);
\draw[step=1,green,thick] (1,1) grid (2,2);
\draw[step=1,red,thick] (3,0) grid (4,1);
\end{tikzpicture}
}

\newcommand{\ND}{
\begin{tikzpicture}[x=1.5em,y=1.5em,thick,color = blue]
\draw[step=1,gray,thin] (0,0) grid (4,2);
\draw[color=black, thick, sharp corners] (0,0) rectangle (4,2);
\draw(1.0,1.5)--(0.5,1.5)--(0.5,1.0);
\draw(1.5,2.0)--(1.5,1.5)--(1.0,1.5);
\node at (1.5,1.5) {$\bullet$};
\draw(3.0,1.5)--(2.5,1.5)--(2.5,1.0);
\draw(3.5,2.0)--(3.5,1.0);
\draw(4.0,1.5)--(3.0,1.5);
\draw(0.5,1.0)--(0.5,0.5);
\draw(2.0,0.5)--(1.5,0.5)--(1.5,0.0);
\draw(2.5,1.0)--(2.5,0.0);
\draw(3.0,0.5)--(2.0,0.5);
\draw(3.5,1.0)--(3.5,0.5)--(3.0,0.5);
\draw[step=1,green,thick] (1,1) grid (2,2);
\draw[step=1,red,thick] (2,0) grid (3,1);
\end{tikzpicture}
}

\newcommand{\NO}{
\begin{tikzpicture}[x=1.5em,y=1.5em,thick,color = blue]
\draw[step=1,gray,thin] (0,0) grid (4,2);
\draw[color=black, thick, sharp corners] (0,0) rectangle (4,2);
\draw(1.0,1.5)--(0.5,1.5)--(0.5,1.0);
\draw(1.5,2.0)--(1.5,1.5)--(1.0,1.5);
\node at (1.5,1.5) {$\bullet$};
\draw(2.5,2.0)--(2.5,1.0);
\draw(3.5,1.5)--(3.5,1.0);
\draw(0.5,1.0)--(0.5,0.5);
\draw(2.0,0.5)--(1.5,0.5)--(1.5,0.0);
\draw(2.5,1.0)--(2.5,0.0);
\draw(3.0,0.5)--(2.0,0.5);
\draw(3.5,1.0)--(3.5,0.5)--(3.0,0.5);
\draw[step=1,green,thick] (1,1) grid (2,2);
\draw[step=1,red,thick] (3,0) grid (4,1);
\node[color=black] at (1.5,-.508) {$c$};
\node[color=black] at (3.5,-.4) {$c'$};
\end{tikzpicture}

}

\newcommand{\daa}{
\begin{tikzpicture}[x=1.5em,y=1.5em,thick,color = blue]
\draw[step=1,gray,thin] (0,0) grid (4,2);
\draw[color=black, thick, sharp corners] (0,0) rectangle (4,2);
\draw(1.0,1.5)--(0.5,1.5)--(0.5,1.0);
\draw(1.5,2.0)--(1.5,1.0);
\draw(2.0,1.5)--(1.0,1.5);
\draw(2.5,2.0)--(2.5,1.0);
\draw(3.0,1.5)--(2.0,1.5);
\draw(4.0,1.5)--(3.5,1.5)--(3.0,1.5);
\draw(0.5,1.0)--(0.5,0.5);
\draw(1.5,1.0)--(1.5,0.0);
\draw(2.5,1.0)--(2.5,0.0);
\draw(4.0,0.5)--(3.5,0.5)--(3.5,0.0);
\draw[step=1,red,thick] (3,0) grid (4,1);
\draw[step=1,green,thick] (0,1) grid (1,2);
\end{tikzpicture}
}

\newcommand{\dab}{
\begin{tikzpicture}[x=1.5em,y=1.5em,thick,color = blue]
\draw[step=1,gray,thin] (0,0) grid (4,2);
\draw[color=black, thick, sharp corners] (0,0) rectangle (4,2);
\draw(1.0,1.5)--(0.5,1.5)--(0.5,1.0);
\draw(1.5,2.0)--(1.5,1.0);
\draw(2.0,1.5)--(1.0,1.5);
\draw(2.5,2.0)--(2.5,1.0);
\draw(3.0,1.5)--(2.0,1.5);
\draw(4.0,1.5)--(3.5,1.5)--(3.0,1.5);
\draw(0.5,1.0)--(0.5,0.0);
\draw(1.0,0.5)--(0.0,0.5);
\draw(1.5,1.0)--(1.5,0.5)--(1.0,0.5);
\draw(2.5,1.0)--(2.5,0.0);
\draw(4.0,0.5)--(3.5,0.5)--(3.5,0.0);
\draw[step=1,red,thick] (3,0) grid (4,1);
\draw[step=1,green,thick] (1,1) grid (2,2);
\end{tikzpicture}
}

\newcommand{\dac}{
\begin{tikzpicture}[x=1.5em,y=1.5em,thick,color = blue]
\draw[step=1,gray,thin] (0,0) grid (4,2);
\draw[color=black, thick, sharp corners] (0,0) rectangle (4,2);
\draw(1.0,1.5)--(0.5,1.5)--(0.5,1.0);
\draw(1.5,2.0)--(1.5,1.5)--(1.0,1.5);
\draw(2.5,2.0)--(2.5,1.0);
\draw(4.0,1.5)--(3.5,1.5)--(3.5,1.0);
\draw(0.5,1.0)--(0.5,0.5);
\draw(2.0,0.5)--(1.5,0.5)--(1.5,0.0);
\draw(2.5,1.0)--(2.5,0.0);
\draw(3.0,0.5)--(2.0,0.5);
\draw(3.5,1.0)--(3.5,0.0);
\draw(4.0,0.5)--(3.0,0.5);
\draw[step=1,red,thick] (3,0) grid (4,1);
\draw[step=1,green,thick] (1,1) grid (2,2);
\end{tikzpicture}
}

\newcommand{\dba}{
\begin{tikzpicture}[x=1.5em,y=1.5em,thick,color = blue]
\draw[step=1,gray,thin] (0,0) grid (4,2);
\draw[color=black, thick, sharp corners] (0,0) rectangle (4,2);
\draw(1.0,1.5)--(0.5,1.5)--(0.5,1.0);
\draw(1.5,2.0)--(1.5,1.0);
\draw(2.0,1.5)--(1.0,1.5);
\draw(3.0,1.5)--(2.0,1.5);
\draw(3.5,2.0)--(3.5,1.0);
\draw(4.0,1.5)--(3.0,1.5);
\draw(0.5,1.0)--(0.5,0.5);
\draw(1.5,1.0)--(1.5,0.0);
\draw(3.0,0.5)--(2.5,0.5)--(2.5,0.0);
\draw(3.5,1.0)--(3.5,0.5)--(3.0,0.5);
\node at (3.5,0.5) {$\bullet$};
\draw[step=1,red,thick] (2,0) grid (3,1);
\draw[step=1,green,thick] (0,1) grid (1,2);
\end{tikzpicture}
}

\newcommand{\dbb}{
\begin{tikzpicture}[x=1.5em,y=1.5em,thick,color = blue]
\draw[step=1,gray,thin] (0,0) grid (4,2);
\draw[color=black, thick, sharp corners] (0,0) rectangle (4,2);
\draw(1.0,1.5)--(0.5,1.5)--(0.5,1.0);
\draw(1.5,2.0)--(1.5,1.0);
\draw(2.0,1.5)--(1.0,1.5);
\draw(3.0,1.5)--(2.0,1.5);
\draw(3.5,2.0)--(3.5,1.0);
\draw(4.0,1.5)--(3.0,1.5);
\draw(0.5,1.0)--(0.5,0.0);
\draw(1.0,0.5)--(0.0,0.5);
\draw(1.5,1.0)--(1.5,0.5)--(1.0,0.5);
\draw(3.0,0.5)--(2.5,0.5)--(2.5,0.0);
\draw(3.5,1.0)--(3.5,0.5)--(3.0,0.5);
\node at (3.5,0.5) {$\bullet$};
\draw[step=1,red,thick] (2,0) grid (3,1);
\draw[step=1,green,thick] (1,1) grid (2,2);
\end{tikzpicture}
}

\newcommand{\dbc}{
\begin{tikzpicture}[x=1.5em,y=1.5em,thick,color = blue]
\draw[step=1,gray,thin] (0,0) grid (4,2);
\draw[color=black, thick, sharp corners] (0,0) rectangle (4,2);
\draw(1.0,1.5)--(0.5,1.5)--(0.5,1.0);
\draw(1.5,2.0)--(1.5,1.5)--(1.0,1.5);
\draw(3.0,1.5)--(2.5,1.5)--(2.5,1.0);
\draw(3.5,2.0)--(3.5,1.0);
\draw(4.0,1.5)--(3.0,1.5);
\draw(0.5,1.0)--(0.5,0.5);
\draw(2.0,0.5)--(1.5,0.5)--(1.5,0.0);
\draw(2.5,1.0)--(2.5,0.0);
\draw(3.0,0.5)--(2.0,0.5);
\draw(3.5,1.0)--(3.5,0.5)--(3.0,0.5);
\node at (3.5,0.5) {$\bullet$};
\draw[step=1,red,thick] (2,0) grid (3,1);
\draw[step=1,green,thick] (1,1) grid (2,2);
\end{tikzpicture}
}

\newcommand{\dca}{
\begin{tikzpicture}[x=1.5em,y=1.5em,thick,color = blue]
\draw[step=1,gray,thin] (0,0) grid (4,2);
\draw[color=black, thick, sharp corners] (0,0) rectangle (4,2);
\draw(1.0,1.5)--(0.5,1.5)--(0.5,1.0);
\draw(1.5,2.0)--(1.5,1.0);
\draw(2.0,1.5)--(1.0,1.5);
\draw(2.5,2.0)--(2.5,1.0);
\draw(3.0,1.5)--(2.0,1.5);
\draw(3.5,1.5)--(3.0,1.5);
\draw(0.5,1.0)--(0.5,0.5);
\draw(1.5,1.0)--(1.5,0.0);
\draw(2.5,1.0)--(2.5,0.0);
\draw[step=1,red,thick] (3,0) grid (4,1);
\draw[step=1,green,thick] (0,1) grid (1,2);
\node[color=black] at (3.5,-.508) {$c$};
\node[color=black] at (0.5,-.4) {$\lambda$};
\end{tikzpicture}
}

\newcommand{\dcb}{
\begin{tikzpicture}[x=1.5em,y=1.5em,thick,color = blue]
\draw[step=1,gray,thin] (0,0) grid (4,2);
\draw[color=black, thick, sharp corners] (0,0) rectangle (4,2);
\draw(1.0,1.5)--(0.5,1.5)--(0.5,1.0);
\draw(1.5,2.0)--(1.5,1.0);
\draw(2.0,1.5)--(1.0,1.5);
\draw(2.5,2.0)--(2.5,1.0);
\draw(3.0,1.5)--(2.0,1.5);
\draw(3.5,1.5)--(3.0,1.5);
\draw(0.5,1.0)--(0.5,0.0);
\draw(1.0,0.5)--(0.0,0.5);
\draw(1.5,1.0)--(1.5,0.5)--(1.0,0.5);
\draw(2.5,1.0)--(2.5,0.0);
\draw[step=1,red,thick] (3,0) grid (4,1);
\draw[step=1,green,thick] (1,1) grid (2,2);
\node[color=black] at (3.5,-.508) {$c$};
\node[color=black] at (1.5,-.4) {$\lambda$};
\end{tikzpicture}
}

\newcommand{\dcc}{
\begin{tikzpicture}[x=1.5em,y=1.5em,thick,color = blue]
\draw[step=1,gray,thin] (0,0) grid (4,2);
\draw[color=black, thick, sharp corners] (0,0) rectangle (4,2);
\draw(1.0,1.5)--(0.5,1.5)--(0.5,1.0);
\draw(1.5,2.0)--(1.5,1.5)--(1.0,1.5);
\draw(2.5,2.0)--(2.5,1.0);
\draw(3.5,1.5)--(3.5,1.0);
\draw(0.5,1.0)--(0.5,0.5);
\draw(2.0,0.5)--(1.5,0.5)--(1.5,0.0);
\draw(2.5,1.0)--(2.5,0.0);
\draw(3.0,0.5)--(2.0,0.5);
\draw(3.5,1.0)--(3.5,0.5)--(3.0,0.5);
\node at (3.5,0.5) {$\bullet$};
\node[color=black] at (3.5,-.508) {$c$};
\node[color=black] at (1.5,-.4) {$\lambda$};
\draw[step=1,red,thick] (3,0) grid (4,1);
\draw[step=1,green,thick] (1,1) grid (2,2);
\end{tikzpicture}
}

\begin{figure}[b]
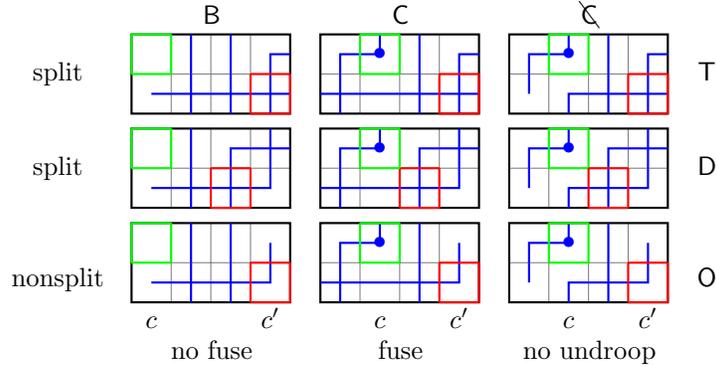

\[
\begin{array}{ccccc}
&\Blank & \Cross & \NCross & \\
\text{split}& \boxit{\BT} & \boxit{\CT} & \boxit{\NT}& \Terminal  \\[5mm]
\text{split}&\boxit{\BD} &\boxit{\CD} &\boxit{\ND}& \DCross \\[5mm]
\text{nonsplit}&\boxit{\BO} &\boxit{\CO} &\boxit{\NO}& \Ordinary  \\[5mm]
 &\text{no fuse} & \text{fuse} &\text{no undroop}
\end{array}
\]
\caption{Cases for $f$ and $f^*$}
\label{F:f 9 cases}
\end{figure}

\begin{figure}[t]
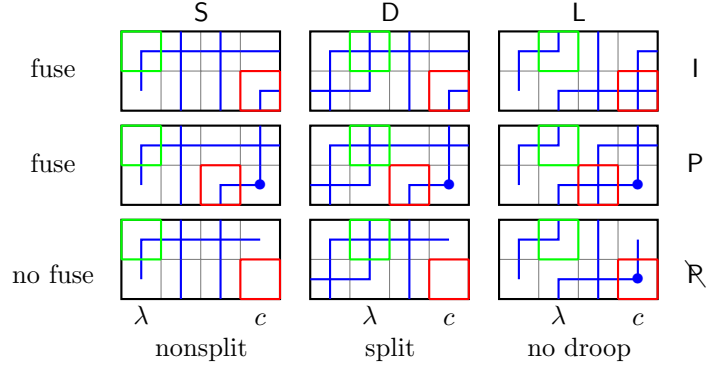

\[
\begin{array}{ccccc}
&\Straight & \DCross& \Leftturn& \\
\text{fuse}& \boxit{\daa} & \boxit{\dab} & \boxit{\dac}& \Initial \\[5mm]
\text{fuse}&\boxit{\dba} &\boxit{\dbb} &\boxit{\dbc}& \Plus \\[5mm]
\text{no fuse}&\boxit{\dca} &\boxit{\dcb} &\boxit{\dcc}& \NPlus \\[5mm]
&\text{nonsplit} & \text{split}& \text{no droop}
\end{array}
\]
\caption{Cases for $e$ and $e^*$}
\label{F:e 9 cases}
\end{figure}

\section{$\ftarg$-moves and $\etarg$-moves are inverses}
In this section, we verify that $\ftarg$-moves and $\etarg$-moves are inverses.
More specifically, 
we establish Proposition~\ref{P: EF} and Proposition~\ref{P: FE}.
\begin{prop}
\label{P: EF}
Let $D$ have $F$-target $(r,c)$
with $\window$ $[r, r+1] \times [b, c']$.
\begin{itemize}
\item Then $F_r(D)$ has $\etarg$-target $(r+1, c')$
with the same window.
\item The case of $(D,(r,c))$ 
corresponds
to the case of $(F_r(D),(r+1,c'))$ 
according to Figure~\ref{F:f 9 cases}
and Figure~\ref{F:e 9 cases}.
For instance, 
if $(D,(r,c))$ is in case $\Blank\Terminal$,
then $(F_r(D),(r+1,c'))$ is in case $\Straight\Initial$.
\item We have $E_r (F_r(D)) = D$.
\end{itemize}
\end{prop}
\begin{proof}
Let $\lambda = c$ and $\rho$ be the left and right
droop columns of $(D,(r,c))$.
We start by proving a useful claim.
\begin{claim}
\label{Claim(EF): top is ps}
If $(D,(r,c))$ is not in Case $\NCross$
then $F_r(D)_{r, [b, c']}$ is a pipe segment.
\end{claim}
\begin{proof}
We first check $F_r(D)_{r, [b, c]}$ 
is a pipe segment.
If $(D,(r,c))$ is in case $\Cross$,
then $D_{r, [b, c]}$ is a pipe segment,
and $F_r(D)_{r, [b, c]}$ is also.
Otherwise
we have $b = c$ and $F_r(D)_{r, [b, c]}$ is 
trivially a pipe segment. 

Since $F_r$ performs an $(r, [c, \rho])$-undroop,
$F_r(D)_{r, [c, \rho]}$ is a pipe segment. 
By Lemma~\ref{L: Combine pipe segment},
it remains to check $F_r(D)_{r, [\rho, c']}$
is a pipe segment.
If $(D,(r,c))$ is not in case $\DCross$,
$\rho = c'$ and $F_r(D)_{r, [\rho, c']}$
is trivially a pipe segment.
Otherwise $D_{[r, r+1], [\rho, c']}$
is a double crossing. 
Thus $D_{r, [\rho, c']}$ is a pipe segment, 
and $F_r(D)_{r, [\rho, c']}$ is also.
\end{proof}

Now we check $\Y = (F_r(D),(r+1,c'))$ is an $E$-target.
\begin{itemize}
\item[(e1)]
We suppose $(D,(r,c))$ is not in Case $\Terminal$ and show $\Y$ satisfies (e1).
By (f3) and Lemma~\ref{L: Only heavy after F},
all tiles in $F_r(D)_{r+1, [1, c')}$ are light.
By Lemma~\ref{L: Only heavy after F},
$F_r(D)_{r+1, c'}$ is heavy, 
so it satisfies (e1).

\item[(e*1)]
We assume $(D,(r,c))$ is in case $\Terminal$ and show $\Y$ satisfies (e*1).
By (f3) and Lemma~\ref{L: Only heavy after F},
all tiles in $F_r(D)_{r+1, [1, c')}$ are light.
By Lemma~\ref{L:TDO} \eqref{I:bottom right} $D_{r+1, c'} = \ptile$.
By (f*2) there is neither $\jtile$ nor $\mtile$ in $D_{r+1,(c,n]}$.
It follows that there are no heavy tile in $D_{r+1, [c', n]}$,
or equivalently in $F_r(D)_{r+1, [c', n]}$.
Thus, there are no heavy tiles in row $r+1$ 
of $F_r(D)$.
We know there are no $\jtile$ 
in $D_{[r, r+1], (c', n]}$; we already checked this for row $r+1$,
and for row $r$ it holds by the choice of $c'$ in (f*2), because
somewhere to the right of every $\jtile$ there must be an $\rtile$.
Therefore there is neither $\jtile$ nor $\rtile$ in 
$F_r(D)_{[r, r+1], (c', n]}$.
It remains to check $F_r(D)_{r, c'}$
or $F_r(D)_{r+1, c'}$ is $\rtile$.
By (f*2), $D_{r,c'}=\rtile$.
If $(D,(r,c))$ is in Case $\NCross$,
then $F_r$ does no undroop and
$F_r(D)_{r,c'}=D_{r,c'}=\rtile$.
Otherwise, after the $(r, [c, \rho])$-undroop,
$(r+1, c')$ changes from $\ptile$ to $\rtile$.
\item[(e2)] We check $F_r(D)_{r, b}$ is the 
rightmost tile satisfying the condition of (e2) 
for $\Y$.
First, we check $F_r(D)_{r,b} = \rtile$:
In case $\Blank$, 
$b = c$ and $(r,b)$ becomes $\rtile$
after the undroop;
Otherwise, 
$D_{r, b} = \rtile$
and stays unchanged by $F_r$.

Next, we check $F_r(D)_{r+1, [b, c']}$
is not a pipe segment.
In case $\NCross$,
$D_{r+1, [b,c']}$ is not a pipe segment
since $D_{r+1,[b,c]}$ is not,
and there is no undroop.
In case $\Blank$ or $\Cross$,
$F_r(D)_{r+1,\rho}$, being the bottom right corner after an undroop,
must be $\rtile$ or $\btile$, both of which cannot 
be part of a pipe segment.

Finally, take $b < h < c'$ with
$F_r(D)_{r, h} = \rtile$.
We need to verify 
$F_r(D)_{r+1,[h,c']}$ is a pipe segment.
If we are not in Case $\NCross$,
by Claim~\ref{Claim(EF): top is ps},
there is no $\rtile$ in $F_r(D)_{r, (b, c']}$,
so such $h$ cannot exist.
Otherwise, 
$D_{r, [b, c]}$ is a pipe segment, 
so $c < h < c'$.
Then since $D_{r+1, [c, c']}$ is a pipe segment by Lemma~\ref{L:TDO} \eqref{I:pipe}, 
so is $D_{r+1, [h, c']}$. In Case $\NCross$ there is no undroop
so the pipe segment still exists after $F_r$.

\item[(e3)] 
By (f1) for $(D,(r,c))$,  
there are no heavy tiles in $D_{r, (c', n]}$,
or equivalently $F_r(D)_{r, (c', n]}$.
By Lemma~\ref{L: Only heavy after F},
there are no heavy tiles in $D_{r, [b, c']}$.
\end{itemize}

Thus, $\Y$ is an $E$-target with window $[r, r+1] \times [b, c']$.
Next we go through all the cases of $(D,(r,c))$,
assuming they hold and proving that the 
corresponding case of $\Y$ holds. For example, in the proof
 $\DCross\rightarrow\Plus$, we assume $(D,(r,c))$ is in Case $\DCross$
and prove that $\Y$ is in Case $\Plus$.
\begin{itemize}
\item[ $\Terminal\rightarrow\Initial$:]
Already checked above. In fact it was shown above that $\Terminal \leftrightarrow \Initial$.
\item[ $\DCross\rightarrow\Plus$:]
We already know $\Initial$ does not hold.
By Lemma~\ref{L:TDO} \eqref{I:light}
$D_{r, c'} = \ptile$.
Whether or not $F_r$ performs an
$(r, [c, \rho])$-undroop,
$F_r(D)_{r, c'}$ is still $\ptile$
since $\rho < c'$.
\item[ $\Ordinary\rightarrow\NPlus$:]
Again we know that $\Initial$ does not hold.
By Lemma~\ref{L:TDO} \eqref{I:light}, 
$D_{r,c'}\ne \ptile$.
Whether or not $F_r$ performs an
$(r, [c, c'])$-undroop,
$F_r(D)_{r, c'} \ne \ptile$.
\item[ $\Blank\rightarrow\Straight$:]
In this case $c = b$ and $F_r$ does an $(r, [c, \rho])$-undroop,
so $F_r(D)_{r+1, c} \neq \ptile$.
Thus, $F_r(D)_{r, c}$ cannot be the top-left
corner of a doublecross.
By Claim~\ref{Claim(EF): top is ps}, 
$F_r(D)_{r, [c, c']}$ 
is a pipe segment, 
so $\Y$ is in case $\Straight$.
\item [ $\Cross\rightarrow\DCross$:] 
By Claim~\ref{Claim(EF): top is ps}, 
$F_r(D)_{r, [b, c']}$ is a pipe segment. 
It is enough to show that $F_r(D)_{[r, r+1], [b, c]}$
is a doublecross.
First, $F_r(D)_{r, [b, c]}$ is a pipe segment
since $F_r(D)_{r, [b, c']}$ is.
By Case $\Cross$, $D_{r+1, [b, c]}$ is a pipe segment.
The $F_r$ performs an $(r, [c, \rho])$-undroop,
so $F_r(D)_{r+1, [b, c]}$ is a pipe segment.
We have shown $F_r(D)_{r, b} = \rtile$.
Finally, 
since $F_r(D)_{r+1, [b, c]}$ is a pipe segment,
$F_r(D)_{r+1, c}$ can be $\jtile$,
$\htile$ or $\ptile$.
Since $F_r(D)$ is obtained by
an $(r, [c, \rho])$-undroop,
$F_r(D)_{r+1, c} = \jtile$.
\item [ $\NCross\rightarrow\Leftturn$:]
We know $D_{r, c} = \mtile$,
so $F_r(D)_{r, c} = \jtile$.
Since $b < c < c'$,
$F_r(D)_{r, [b, c']}$ cannot be a pipe segment.
\end{itemize}

Finally, we verify $E_r(F_r(D)) = D$.
Let $\lambda', \rho'$
be the left and right droop columns of $\Y$.
We may suppose that $(D, (r, c))$ is not in Case $\NCross$
and show that
$\lambda' = c$ and $\rho' = \rho$. 
We prove $\lambda'=c$:
\begin{itemize}
\item Suppose the $F$-target $(D,(r,c))$ is in Case $\Blank$.
We have shown that the $E$-target $\Y$ is in Case $\Straight$.
We have $c=b$ from Case $\Blank$ and $b=\lambda'$ from Case $\Straight$.
\item Suppose the $F$-target $(D,(r,c))$ is in Case $\Cross$.
We know $D_{r+1, [b, c']}$ is a pipe segment. 
After applying the $(r, [c, \rho])$-undroop,
$F_r(D)_{r+1, c} = \jtile$
and $F_r(D)_{r+1, [b, c]}$ is a pipe segment.
Thus, $\lambda' = c$.
\end{itemize}
We prove $\rho'=\rho$:
\begin{itemize}
\item Suppose the $F$-target $(D,(r,c))$ is in Case $\DCross$.
We know $D_{r+1, [\rho, c']}$ is a pipe segment
and $D_{r+1, \rho} = \ptile$.
After applying the $(r, [c, \rho])$-undroop,
$F_r(D)_{r+1, [\rho, c']}$ remains a pipe segment
and $F_r(D)_{r+1, c} = \rtile$.
Thus, $\rho' = \rho$.
\item Suppose the $F$-target $(D,(r,c))$ is not in Case $\DCross$.
Then $\rho' = c' = \rho$. \qedhere
\end{itemize}
\end{proof}

\begin{prop}
\label{P: FE}
Let $D$ have $E$-target $(r+1,c)$
with $\window$ $[r, r+1] \times [c', c]$.
Let $\lambda, \rho$
be its left and right droop columns. 
\begin{itemize}
\item Then $E_r(D)$ has $\ftarg$-target $(r, \lambda)$
with the same window.
\item The case of $(D,(r+1,c))$ 
corresponds
to the case of $(E_r(D),(r+1,c'))$ 
according to Figure~\ref{F:f 9 cases}
and Figure~\ref{F:e 9 cases}.
\item We have $F_r (E_r(D)) = D$.
\end{itemize}
\end{prop}
\begin{proof}
We start by proving a useful claim. 

\begin{claim}
\label{claim(FE)}
$E_r(D)_{r+1, [\lambda, c]}$ is a pipe segment.
\end{claim}
\begin{proof}
We first consider Case $\Leftturn$.
It is enough to show $D_{r+1, [\lambda, c]}$
is a pipe segment. 
Let $d < c$ be the largest such that $D_{r+1, d}
= \rtile$.
By Lemma~\ref{L: What can E target be}, 
$D_{r+1, c}$ connects to the left,
so $D_{r+1, [d,c]}$ is a pipe segment. 
It is enough to show $d \leq \lambda$.
Suppose $d>\lambda$.
(e2) implies $d>c'$. 
By (e3) $D_{r,d}$ is light and not connected below,
so it is connected to the left. Going down this pipe to the left,
it must turn downwards at a $\rtile$ before reaching the $\jtile$
at $(r,\lambda)$. This $\rtile$ contradicts the choice of $c'$ in (e2).
Therefore $d\le \lambda$.

Now suppose we are not in Case $\Leftturn$.
Then $E_r(D)$ is obtained from $D$
via a $(r, [\lambda, \rho])$-droop
and possibly marking or unmarking certain tiles. 
Thus $E_r(D)_{r+1, [\lambda, \rho]}$ is a
pipe segment.
By Lemma~\ref{L: Combine pipe segment},
it remains to show $E_r(D)_{r+1, [\rho, c]}$
is a pipe segment. 
If not in Case $\Plus$, 
we are done since $c = \rho$.
If in Case $\Plus$,
we know $D_{r+1, [\rho, c]}$
is a pipe segment, and hence
so is $E_r(D)_{r+1, [\rho, c]}$.
\end{proof}

Now we check $\Y = (E_r(D),(r,\lambda))$ is an $F$-target.
\begin{itemize}
\item[(f1)] 
By (e3) for $(D,(r+1,c))$,  
there are no heavy tiles in $D_{r, (c, n]}$,
or equivalently $E_r(D)_{r, (c, n]}$.
By Lemma~\ref{L: Only heavy after E},
$(r, \lambda)$ is the only heavy tile in
$E_r(D)_{r, [c', c]}$.
Thus, there are no heavy tiles in 
$E_r(D)_{r, (\lambda, n]}$.

\item[(f2)]
Assume $(D,(r+1,c))$ is not in case $\Initial$.
We show $\Y$ satisfies (f2)
by checking $(r+1, c)$ is the leftmost
$\jtile$ in $E_r(D)_{r+1,[\lambda, n]}$.
By Claim~\ref{claim(FE)}, we know there
are no $\jtile$ in $E_r(D)_{r+1,[\lambda, c]}$.
Then by Lemma~\ref{L: Only heavy after E},
$E_r(D)_{r+1,c} = \jtile$.

\item[(f*2)]
Assume $(D,(r+1,c))$ is in case $\Initial$.
We show $\Y$ satisfies (f*2)
and verify $E_r(D)_{r, c}$ is
the rightmost $\rtile$ in row $r$ of $E_r(D)$.
By Claim~\ref{claim(FE)}, 
there are no $\jtile$ in 
$E_r(D)_{r+1, [\lambda, c)}$.
It remains to check $E_r(D)_{r,c}$
is the only $\rtile$ in
in $E_r(D)_{[r, r+1], [c, n]}$.
This will imply there are no 
$\jtile$ in $E_r(D)_{r+1, [c, n]}$

By (e*1) of $(r+1, c)$ in $D$,
we know $(r, c)$ or $(r+1, c)$
is the only $\rtile$
in $D_{[r, r+1], [c, n]}$.
If $D_{r,c} = \rtile$,
we know $D_{r, [c', c]}$ is not a pipe segment,
so we are in Case $\Leftturn$.
Then $E_{r}(D)_{r, c}$ remains the only $\rtile$
in $E_r(D)_{[r, r+1], [c, n]}$.
If $D_{r + 1,c} = \rtile$,
we know $D_{r, c} = \htile$.
After the $(r, [\lambda, c])$-droop,
$D_{r, c}$ becomes the only $\rtile$
in $E_r(D)_{[r, r+1], [c, n]}$.

\item[(f3)] 
By (e1) or (e*1) for $(D,(r+1,c))$,  
there are no heavy tiles in $D_{r+1, [1, c')}$,
or equivalently $E_r(D)_{r, [1, c')}$.
By Lemma~\ref{L: Only heavy after E},
there are no heavy tiles in $E_r(D)_{r+1, [c', c]}$.
\end{itemize}

Thus, $\Y$ is an $F$-target.
Its window is $[r, r+1] \times [b, c]$
for some $b$.
Next we go through all the cases of $(D,(r+1,c))$,
assuming they hold and proving that the 
corresponding case of $\Y$ holds. 
For example, in the proof
$\Plus\rightarrow\DCross$, we assume $(D,(r+1,c))$ is in Case $\Plus$
and prove that $\Y$ is in Case $\DCross$.
In the last three proofs, we also check the window of $\Y$
is the same as $(D, (r+1, c))$ by showing $b = c'$.

\begin{itemize}
\item[ $\Initial \rightarrow \Terminal$:]
Already checked above. In fact it was shown above that $\Initial \leftrightarrow \Terminal $.
\item[ $\Plus \rightarrow \DCross$:]
We already know $\Terminal$ does not hold.
By the definition of Case $\Plus$,
$D_{r, c} = \ptile$.
Whether or not $E_r$ performs an
$(r, [\lambda, \rho])$-droop,
$E_r(D)_{r, c}$ is still $\ptile$
since $\rho < c$.
By Lemma~\ref{L:TDO} \eqref{I:light}, 
$\Y$ is in Case $\DCross$.
\item[ $\NPlus\rightarrow\Ordinary$:]
Again we know that $\Terminal$ does not hold.
By the definition of Case $\NPlus$,
$D_{r, c} \neq \ptile$.
Whether or not $E_r$ performs an
$(r, [\lambda, \rho])$-droop,
$E_r(D)_{r, c}$ is still not a $\ptile$.
By Lemma~\ref{L:TDO} \eqref{I:light}, 
$\Y$ is in Case $\Ordinary$.
\item[$\Straight\rightarrow\Blank$:]
In this case, 
we have $\lambda = c'$ by definition.
Since $D_{r, \lambda} = \rtile$,
$E_r(r, \lambda) = \btile$
after the $(r, [\lambda, \rho])$-droop.
Thus, $\Y$ is in Case $\Blank$.
In this case, $b$ is defined as $\lambda = c'$.
\item [ $\DCross\rightarrow \Cross$:] 
We know $D_{r, [c', c]}$ is a pipe segment
and $D_{r+1, \lambda} = \jtile$,
so $D_{r, \lambda} = \ptile$.
After making the $(r, [\lambda, \rho])$-droop
and marking $(r, \lambda)$,
we know $E_r(D)_{r, \lambda} = \mtile$.
Since $D_{r, c'} = \rtile$,
we know $E_r(D)_{r, c'} = \rtile$
and $E_r(D)_{r, [c', \lambda]}$
is a pipe segment. 
Thus, $b = c'$.

By Claim~\ref{claim(FE)}, 
$E_r(D)_{r+1, [\lambda,c]}$ is a pipe segment.
Then since $D_{r+1, [c', \lambda]}$
is a pipe segment, 
so is $E_r(D)_{r+1, [c', \lambda]}$.
By Lemma~\ref{L: Combine pipe segment}, 
$E_r(D)_{r+1, [c' ,c]}$ is a pipe segment,
so $\Y$ is in Case $\Cross$.

\item [ $\Leftturn\rightarrow\NCross$:]
We know $D_{r, \lambda} = \jtile$
and is marked by $E_r$,
so $E_r(D)_{r, \lambda} = \mtile$.
By how we defined $\lambda$,
$D_{r, [c', \lambda]}$ is a pipe
segment. $E_r(D)_{r, [c', \lambda]}$ remains a pipe
segment and $E_r(D)_{r, c'} = \rtile$,
so $b = c'$.

By condition (e2),
$D_{r+1, [c' ,c]}$ is not a pipe segment.
Thus, $E_r(D)_{r+1, [c', c]}$
is not a pipe segment and $\Y$ is in case $\NCross$.
\end{itemize}

Finally, we check $F_r(E_r(D)) = D$.
We may assume $(D,(r+1,c))$ is not in case $\Leftturn$.
Let $\lambda'$ and $\rho'$ be the left and right droop columns of $\Y$.

We first check that $\lambda'=\lambda$.
\begin{itemize}
\item Suppose $(D,(r+1,c))$ is in case $\Straight$.
Then $\lambda=c'$. $\Y$ is in case $\Blank$ so 
$\lambda'=c'$ is the column of the $\btile$.
\item Suppose $(D,(r+1,c))$ is in case $\DCross$.
Then $\lambda>c'$ is minimum such that $D_{r+1,\lambda}$ is $\jtile$.
$E_r(D)_{r,\lambda}=\mtile$ so that $\lambda'=\lambda$ is the column of the $\mtile$.
\end{itemize}
We check that $\rho'=\rho$.
\begin{itemize}
\item Suppose $(D,(r+1,c))$ is in Case $\Initial$ or $\NPlus$.
Then $\rho=c$. For $\Y$ we are in Case $\Terminal$ or $\Ordinary$ respectively, and in both cases we have $\rho'=c$, the right end column index of the window.
\item Suppose $(D, (r+1, c)$ is in Case $\Plus$.
We just need to check 
$[r, r+1] \times [\rho, c]$
is a double crossing in $E_r(D)$.
Since $(D, (r+1, c)$ is not in Case $\Leftturn$,
$D_{r, [c', c]}$ is a pipe segment,
so $E_r(D)_{r, [\rho, c]}$ remains a pipe segment.
By Claim~\ref{claim(FE)},
$E_r(D)_{r+1, [\rho, c]}$ is a pipe segment. 
By Lemma~\ref{L: Only heavy after E},
$E_r(D)_{r+1, c} = \jtile$.
Finally, 
after the $(r, [\lambda, \rho])$-droop,
$E_r(D)_{r, \rho}$ can be $\rtile$ or $\vtile$.
Since $E_r(D)_{r, [\rho, c]}$ is a pipe segment,
$E_r(D)_{r, \rho} = \rtile$. \qedhere
\end{itemize}
\end{proof}

\section{Bijection between marked bumpless pipedreams and compatible pairs}

\subsection{$\Phi:\MBPD(n)\to \RCP(n)$ via $F$-moves.}
\label{SS:Phi via f moves}
The following two simple steps may be iterated to reduce any $B\in\RCP(n)$ to the empty biword. 
Suppose $B$ is nonempty; let $(i,a)$ its first biletter.
\begin{enumerate}
    \item[] (RCP1) If $i=a$ remove the first biletter $(i,i)$ from $B$.
    Write $\XX_i(B)$ for the resulting biword.
    \item[] (RCP2) If $i<a$, replace the first biletter of $B$ by $(i+1,a)$.
    Call $\uparrow B$ the result of this operation.
\end{enumerate}
We define the left inverse operations.
\begin{enumerate}
\item[] (RRCP1) Let $\II_i(B)$ be the result of prepending $(i,i)$ to $B$.
\item[] (RRCP2) If $(i,a)$ is the first biletter of $B$ and $i>1$, replace
$(i,a)$ by $(i-1,a)$. Call the result of this operation $\downarrow B$.
\end{enumerate}

The bijection $\Phi:\MBPD(n)\to\RCP(n)$ is defined by specifying the operations on marked bumpless pipedreams which correspond to the above operations on biwords.
\begin{enumerate}
    \item[(0)] If $D=D_\id$ then $\Phi(D)=()$ is the empty sequence of biletters.
\end{enumerate}
Otherwise $D$ has a heavy tile. Let $D$ have maximum $F$-target $(i,j)$.
By inductive hypothesis 
$\Phi(D)$ has the form $((i,a),(i_2,a_2),\dotsc)$ for some $a\ge i$.
\begin{enumerate}
    \item[(1)] Suppose $D$ is $F$-terminal. Then define $\Phi(D) = \II_i \Phi(\fs_i(D))$.
    \item[(2)] Suppose $D$ is $F$-nonterminal. Then $f_i(D)$ has a single heavy tile in row $i+1$ and none in later rows, so its maximum $F$-target is in row $i+1$. By induction we have defined $\Phi(f_i(D)) = ((i+1,a),(i_2,a_2),(i_3,a_3),\dotsc)$ with $i+1\le a$. 
    Define $\Phi(D) = \downarrow \Phi(f_i(D))=((i,a),(i_2,a_2),\dotsc)$.
\end{enumerate}

$\MBPD(n)$ is the disjoint union of $\{D_\id\}$, the singleton containing the unique MBPD with no heavy tiles, the set $\MBPD(n)_\term$ of MBPDs whose maximal $F$-target is terminal,
and the set $\MBPD(n)_\nonterm$ of MBPDs whose maximal $F$-target is nonterminal. 
We assume the maximal $F$-target is in row $i$. The terminal case is given by the following commutative diagram.
\[
\begin{tikzcd}
  \MBPD(n)_\term \arrow[r,"\Phi"] \arrow[d,swap,"f_i^*"] & \RCP(n) \\
  \MBPD(n) \arrow[r,swap,"\Phi"] & \RCP(n) \arrow[u,swap,"\text{$\II_i$}"] 
\end{tikzcd}
\]
The nonterminal case is given by the diagram
\[
\begin{tikzcd}
  \MBPD(n)_\nonterm \arrow[r,"\Phi"] \arrow[d,swap,"f_i"] & \RCP(n) \\
  \MBPD(n) \arrow[r,swap,"\Phi"] & \RCP(n) \arrow[u,swap,"\downarrow"] 
\end{tikzcd}
\]

We group several $f$ and $f^*$ moves into a single operation called ``row pop''. 
Then the computation of $\Phi(D)$ can be broken into a sequence of row pops.
Starting with $D$ with maximum target in row $i$, we apply $f_i$, then $f_{i+1}$, and so on, as long as the result is not terminal. For some $a \ge i$ the result $f_{a-1} \dotsm f_{i+1}f_i(D)$ is terminal. Then we apply $\fs_a$. Let $\nabla_r D=\fs_a f_{a-1}\dotsm f_{i+1} f_i (D)$. Then 
$\Phi(D) = (i,a) \cdot \Phi(\nabla_r D)$ which means prepending the biletter $(i,a)$ to the biword $\Phi(\nabla_r D)$. We introduce the notation $\rpop(D)=(i,a)$ for this biletter.
We call the map $D\mapsto (\rpop(D),\nabla_r D)$ 
\textit{row pop}. 
Iterating row pop produces the biletters of $\Phi(D)$ from left to right.

Let $\mathbb{B}$ be the set of biletters. We have a commutative diagram
\[
\begin{tikzcd}
  \MBPD(n)\setminus \{D_\id\} \arrow[r,dashed,"\Phi"] \arrow[d,swap,"\rpop\times \nabla_r"] & 
\RCP(n) \setminus \emptyset \\
\mathbb{B}\times \MBPD(n) 
\arrow[r,swap,"\id_{\mathbb{B}}\times\Phi"] & 
   \mathbb{B}\times  \RCP(n)  \arrow[u,swap,"\text{prepend}"]
\end{tikzcd}
\]

\begin{rem} \label{R:needed for bijection} 
To define the bijection $\Phi:\MBPD(n)\to\RCP(n)$ we only need to define moves using the maximum $F$-target (see \S \ref{SS:target}). However to prove that the image biwords satisfy the compatibility condition, we find it convenient to use more general moves, which are defined in terms of $F$-targets which need not be maximum.
\end{rem}

\subsection{Well-definedness of $\Phi$} 
We first prepare a   lemma that addresses a commutativity property of $f$-moves.

\begin{lem}
\label{lem:commute}
If $D$ has an $F$-target in both
row $i$ and $i+1$,
then $F_iF_{i+1}(D) = F_{i+1}F_{i}(D)$.
\end{lem}

\begin{proof}
    We note that the two moves only look at and make changes within their own windows, which are necessarily disjoint by the definition of $F$-target. 
\end{proof}

\begin{prop}
The map $\Phi:\MBPD(n)\to \RCP(n)$ is well-defined and weight-preserving. 
\end{prop}
\begin{proof}
Let $D\in\MBPD(n)$. The proof proceeds by induction on the number of heavy tiles in $D$ and then on the number of operators $f_k$ and $f_k^*$ that must be applied to compute $\Phi(D)$. 

We may assume $D$ has at least one heavy tile. Let $D$ have $F$-target at $(i,j)$. 

For every move $f_r$ in the definition of $\Phi$, the new heavy tile created by applying $f_r$ is always southeast of the target. Thus the new target has fewer $\jtile$'s in the southeast. It follows that the terminal case must eventually be reached: $D$ admits a unique sequence of operators of the form $\fs_a f_{a-1}\dotsm f_{i+1} f_i$ for some $a\ge i$.  Moreover all moves do not disturb heavy tiles in previous rows or columns. 

Suppose $D$ has only one heavy tile. Then after applying $F$, $F(D)$ has no heavy tiles:
$F(D)=D_\id$. There is a unique diagram which $\fs_a$ sends to $D_\id$, namely, the Rothe BPD $D_{s_a}$ for the simple reflection $s_a$.
There is also a unique diagram that $f_{a-1}$ sends to $D_{s_a}$. Working backwards there is one possibility for $D$: a single blank at $(i,i)$, with $\jtile$'s at $(j,j)$ for $i<j\le a$. We get $\Phi(D)=((i,a))$. 

We now assume $D$ has at least two heavy tiles. $\Phi(D)$ must start with $(i,a)$ for some $a\ge i$. Suppose first $D$ only has one heavy tile in row $i$. Let $i'<i$ be maximum such that $D$ has a heavy tile in row $i'$. This tile is unchanged during the computation of $\nabla_r D = \fs_a f_{a-1}\dotsm f_{i+1}f_i D$. Therefore $\nabla_r D$ has a target in row $i'$. Hence the second biletter in $\Phi(D)$ has the form $(i',a')$ for some $a'\ge i'$. Since $i'<i$ the compatibility condition is satisfied between the first and second biletters. By the induction hypothesis $\Phi(\nabla_r D)\in\RCP(n)$. Therefore $\Phi(D)\in \RCP(n)$.

We may assume $D$ has at least two heavy tiles in row $i$. Let $j'<j$ be maximum such that $D_{i,j'}$ is heavy. During the computation of $\nabla_r D$ the $(i,j')$ tile remains unchanged. It follows that the second biletter of $\Phi(D)$ has the form $(i,a')$ for some $a' \ge i$. We must show $a<a'$. This is equivalent to showing that $D$ admits the operator
\begin{align}\label{E:F operator}
\varphi_{i,a} := f_a f_{a-1} \dotsm f_i \fs_a f_{a-1} \dotsm f_i\qquad\text{for $a\ge i$.}
\end{align}

We first argue that $F_i(D)$ admits $f_i$; that is, $F_i(D)$ with the maximal $F$-target $(i,j')$ is $F$-nonterminal.
Since $D_{i,j'}$ is $\btile$ or $\mtile$, it must be the case that $D_{i+1,j'}$ is $\htile$ or $\rtile$.  If $D_{i+1,[j',n]}$ is a pipe segment, then 
$(D,(i,j))$ must be in case $\Terminal\Blank$ or $\Terminal\Cross$. 
In these cases, 
$\fs_i(D)_{i+1,j}=\jtile$. It follows that $(i,j')$ is the maximal $f$-target in $\fs_i(D)$, and hence $F$-nonterminal. 
Otherwise there exists $k>j'$ smallest such that $D_{i+1,k}=\jtile$.
If $k>j$ then $(D,(i,j))$ is in Case $\DCross$ or $\Ordinary$.
In these cases $f_i(D)_{i+1,j}=\jtile$, which implies $(f_i(D),(i,j'))$ is 
nonterminal.  
If $j'<k<j$ then this tile is not changed by $F_i$ on $D$,
so $F_i(D)$ is $F$-nonterminal. 

The argument above shows that if $D$ is $F$-terminal (the case that $a=i$), then $D$ admits $\varphi_{i,a}=\varphi_{i,i}=f_i\fs_i$.

 Now suppose $D$ is $F$-nonterminal. To prove that $D$ admits $\varphi_{i,a}$ there are two cases depending on whether $f_i(D)$ is terminal or not. Since $f_i f_i D$ has two heavy tiles in row $i+1$ and none in later rows, we may assume that $\varphi_{i+1,a}$ applies to $f_i f_i D$.

\textbf{Case $f_i(D)$ is terminal}. In this case $a=i+1$ and $\varphi_{i,i+1} =f_{i+1} f_i \fs_{i+1} f_i$. 
By induction $\varphi_{i+1,i+1}$ applies to $f_i f_i D$. By Lemma~\ref{lem:commute} we have
\begin{align*}
  \varphi_{i+1,i+1} f_i f_i D &= f_{i+1} \fs_{i+1} f_i f_i D \\
  &= f_{i+1} f_i \fs_{i+1} f_i D = \varphi_{i,i+1} D
\end{align*}
and $\varphi_{i,i+1}$ applies to $D$ as required.

\textbf{Case $f_i(D)$ is nonterminal}. Here $a \ge i+2$. 
By induction $\varphi_{i+1,a}$ applies to $f_i f_i D$. By Lemma~\ref{lem:commute} we have
\begin{align*}
  \varphi_{i+1,a} f_i f_i D &= f_a f_{a-1}\dotsm f_{i+1} \fs_a f_{a-1}\dotsm f_{i+1} f_i f_i D \\
  &= f_a f_{a-1}\dotsm f_{i+1} \fs_a f_{a-1}\dotsm f_i f_{i+1} f_i D \\
  &= f_a f_{a-1}\dotsm f_{i+1} f_i \fs_a f_{a-1}\dotsm  f_{i+1} f_i D \\
  &= \varphi_{i,a} D
\end{align*}
as required, using the obvious commutativity of $f_i$ with operators that don't change rows $i$ and $i+1$.
\end{proof}

\subsection{ $\Psi:\RCP(n)\to\MBPD(n)$ via $E$-moves}
Recall the operations $\II_i$, $\XX_i$, $\uparrow$ and $\downarrow$
on $\RCP(n)$ defined in \S \ref{SS:Phi via f moves}.
Define the map $\Psi:\RCP(n)\to\MBPD(n)$ as follows.
$\Psi$ sends the empty biword to $D_\id$.
Let $B=((i,a),(i_2,a_2),(i_3,a_3),\dotsc)\in\RCP(n)$ be nonempty.
If $i=a$ let $\Psi(B)=\es_i(\Psi(\XX_i(B)))=\es_i(\Psi((i_2,a_2),(i_3,a_3),\dotsc))$.
Otherwise let $\Psi(B)=e_i(\Psi(\uparrow B))=e_i(\Psi((i+1,a),(i_2,a_2),(i_3,a_3),\dotsc))$.
Similar to the situation for $\Phi$, the computation of the bijection $\Psi$ can be broken into coarser operations.
Let $B=((i,a),(i_2,a_2),\dotsc)\in \RCP(n)$.
By induction we have computed $D':=\Psi((i_2,a_2),(i_3,a_3),\dotsc)\in \MBPD(n)$. Then $D:=\Psi(B)=e_i\dotsm e_{a-2} e_{a-1} e_a^*(D')$. The operation going from $D'$ to $D$ is called the \textit{row push} of $(i,a)$ into $D'$.

Let $B=((i_1,a_1),(i_2,a_2),\dotsc,(a_\ell,i_\ell))\in \RCP(n)$. 
Starting with $D_\id\in\MBPD(n)$, performing row push starting with $(a_\ell,i_\ell)$ and iterating through $(i_1,a_1)$, we obtain $\Psi(B)=D\in\MBPD(n)$.

\subsection{Well-definedness of $\Psi$}

\begin{lem} \label{L:easy es}
Suppose $D\in\MBPD(n)$ has no heavy tiles in rows $r$.
\begin{enumerate}
    \item If $D$ has no heavy tiles in row $r+1$, then $\es_r(D)$ is defined.
    \item If $D$ has a single heavy tile $D_{r+1,d}$ in row $r+1$, then $e_r(D)$ is defined.
\end{enumerate}
\end{lem}
\begin{proof} For (1), let $d$ be maximum such that either $D_{r,d}=\rtile$ or
$D_{r+1,d}=\rtile$. We need to verify that
$(D,(r+1,d))$ is an $e^*$-target. The conditions (e*1) and (e3) are straightforward, so we check (e2). 

Suppose $D_{r,d}=\rtile$. Let $\rho$ be maximum such that 
$D_{r+1,\rho}=\rtile$; necessarily $\rho<d$.
$D_{r,\rho}$ is light and does not connect downwards so it must connect to the left. Therefore there is a $d'<\rho$ such that $D_{r,d'}=\rtile$; take $d'$ maximum. 

Suppose $D_{r+1,d}=\rtile$. $D_{r,d}$ is light and does not connect downwards so it connects to the left. Therefore there is a $d'<d$ such that $D_{r,d'}=\rtile$;
take $d'$ maximum.

For (2), we verify that $(D,(r+1,d))$ is an $e$-target and we only need to check (e2).
If $D_{r+1,d}=\btile$, then $D_{r,d}=\htile$ or $\jtile$. Following this pipe to the left, we will find some $d'<d$ maximal such that $D_{r,d'}=\rtile$.
It follows that $(r+1,d)$ is an $e$-target. If $D_{r+1,d}=\mtile$, we follow this pipe to the left and find $d'<d$ such that $D_{r+1,d'}=\rtile$. Then $D_{r,d}=\htile$ or $\jtile$. Following this pipe to the left we will find $d''<d$ maximal such that $D_{r,d''}=\rtile$. 
\end{proof}

\begin{rem} \label{R:es} In Lemma \ref{L:easy es} the \window{} of the $e^*$-move is $[r,r+1]\times [d',d]$. The assumption that there are no heavy tiles in row $r$, was used to rule out heavy tiles in the range $D_{r,(d',d)}$.
\end{rem}

\begin{lem}
\label{lem:e-commute}
    Suppose $(r,c)$ is an $f$-target in $D\in\MBPD(n)$. Then
    \begin{itemize}
        \item[(1)] If there are no heavy tiles in row $r+1$, then row $r+1$ has an $e^*$-target at $(r+1,d)$ for some $d>c$. If $(r,c)$ is also an $e$-target, then $D':=e_{r-1}e^*_r(D)= e^*_re_{r-1}(D)$. 
        \item[(2)] If $(r+1,d)$ is the leftmost heavy tile in row $r+1$ with $d>c$, then it is an $e$-target. If $(r,c)$ is also an $e$-target, then $D':=e_{r-1}e_r(D)= e_re_{r-1}(D)$. 
    \end{itemize}
    For both (1) and (2), suppose $D'_{r-1, c_1}$ and $D'_{r,d_1}$ are the new heavy tiles in row $r-1$ and $r$ as compared to $D$, then $c_1<d_1$.
\end{lem}

\begin{proof}
Since $(r,c)$ is an $f$-target, by (f1), it is the rightmost heavy tile in row $r$; by (f2), there is a minimal $c'>c$ such that $D_{r+1,c'}=\jtile$; by (f3) all $D_{r+1,j}$ are light for $j<c'$.

For (1), we make the same choices as in Lemma \ref{L:easy es} to construct $d$ and $d'$. We notice that necessarily $c<d'$. The property (f1) for the $F$-target $(D,(r,c))$ implies 
that $(D,(r+1,d))$ is an $e^*$-target. 
If $(r,c)$ is also an $e$-target, the \window{} for the $e$-target $(D,(r,c))$ has $(r,c)$ as the bottom right corner, whereas 
the top left corner of the \window{} for the $e$-target $(D,(r+1,d))$ is $(r,d')$
so these two \window s are necessarily disjoint, implying the desired commutativity as well as the last statement, since the heavy targets are moved within the windows. 

For (2), if $D_{r+1,d}=\btile$, then $D_{r,d}$ is $\htile$ or $\jtile$. It cannot be $\mtile$ since $(r,c)$ is the rightmost heavy tile in row $r$. The rest of the argument for this case is similarly to (1) when $D_{r+1,d}=\rtile$. If $D_{r+1,d}=\mtile$, we follow this pipe to the left and find $D_{r+1,d'}=\rtile$ with $d'<d$ largest. The rest of the argument is similar to (1) when $D_{r,d}=\rtile$. 
\end{proof}
\begin{thm}
    The map $\Psi: \RCP(n)\to\MBPD(n)$ is well-defined and weight-preserving. 
\end{thm}
\begin{proof}
    Let $B\in\RCP(n)$. We proceed by induction on the length of $B$, and then on the number of operators $e_k$ and $e_k^*$ that must be applied to compute $\Psi(B)$.

    Suppose $B$ has length 1. If $B=((a,a))$, $\Psi(B)=e_a^*(D_{\id})=D_{s_a}$. If $B=((i,a))$ for $1\le i<a$, then $\Psi(B)=e_ie_{i+1}\cdots e_{a-1}e^*_a$ is the unique bumpless pipedream for $s_a$ with a single blank tile in row $i$.

    We now assume that $B=((i,a), (i_2,a_2),\cdots)$ has length at least 2. Let $B':=((i_2,a_2),\cdots)$ be $B$ without the first biletter.

    Suppose first that $i> i_2$. By the induction hypothesis, $D':=\Psi(B')$ is well-defined and all its heavy tiles are at or above row $i_2$. 

    If $i=a$, we check that the conditions for applying $e^*_a$ on $D'$ must hold. This follows from Lemma~\ref{L:easy es}(1), since
    there are no heavy tiles in row $a$ and $a+1$.

    If $i<a$, we check that the conditions for applying $e_i$ on $\Psi(\uparrow B)$ must hold. By induction hypothesis, $\Psi(\uparrow B)$ is defined, and contains a single heavy tile $D_{i+1,d}$ in row $i+1$, and all other heavy tiles are in rows $\le i_2$. In particular, there are no heavy tiles in row $i$. That $e_i(\Psi(\uparrow B))$ is defined follows from Lemma~\ref{L:easy es}(2).

    Now suppose $i=i_2$.  Since $B\in\RCP$, we must have $i\le a<a_2$. If $i=a$, by the induction hypothesis $\Psi(B')$ is defined, all heavy tiles are in rows $\le i$, and $\Psi(B')=e_i\Psi(((i+1,a_2),\cdots))$. Since $f_i$ and $e_i$ are inverses, $\Psi(B')$ contains an $f$-target in row $i$. Since there are no heavy tiles in row $i+1$, by Lemma~\ref{lem:e-commute}(1), row $i+1$ of $\Psi(B')$ has an $e^*$-target.

    If $i<a$, by induction hypothesis and the definition of $\Psi$, $\Psi(\uparrow B)$ is defined and 
    \[\Psi(\uparrow B)=\textcolor{red}{e_{i+1}e_{i+2}\cdots e_a^*}\textcolor{blue}{e_i\cdots e_{a-1}e_a }D^{a+1}\] 
    for some $D_{a+1}\in\MBPD(n)$, where $D_{a+1}$ has a single heavy tile in row $a+1$ and all other heavy tiles are in rows $\le i$. By the apparent commutativity of the operators indexed by non-adjacent numbers, we have

    \[\Psi(\uparrow B)={\color{red} e_{i+1}}{\color{blue}e_i}{\color{red}e_{i+2}}{\color{blue}e_{i-1}}\cdots {\color{red}e_{a-1}}{\color{blue}e_{a-2}}{\color{red}e_a^*}{\color{blue}e_{a-1}e_a} D^{a+1}.\] 
   
    (Colors are used for visual aid only.)
    Since $f_a$ and $e_a$ are inverses, $e_aD^{a+1}$ has an $f$-target in row $a$, say at $(a,c_a)$. 
    Since there are no heavy tiles below row $a$ in $e_aD^{a+1}$, $e_aD^{a+1}$ has an $e^*$-target at $(a+1, d_{a+1})$ for some $d_{a+1}>c_a$. 
    Since $e_{a-1}e_aD^{a+1}$ is defined, 
    $(a,c_a)$ is also an $e$-target in $e_aD^{a+1}$. Since there are no heavy tiles below row $a$ in $e_aD^{a+1}$, by Lemma~\ref{lem:e-commute}(1), we have \[D^a:=e_a^*e_{a-1}e_aD^{a+1}=e_{a-1}e_a^*e_aD^{a+1}.\] Furthermore, if $e_{a-1}$ moved the heavy tile $(e_aD^{a+1})_{a,c_a}$ to $D^a_{a-1,c_{a-1}}$ and $e_a^*$ created a heavy tile $D^a_{a,d_a}$, then $c_{a-1}<d_a$.

    Now notice $(D^a,(a-1,c_{a-1}))$ is an $f$-target and $(D^a,(a,d_a))$ is an $e$-target. If $i<a-1$ then $(D^a,(a-1,c_{a-1}))$ is also an $e$-target, we can then apply Lemma~\ref{lem:e-commute}(2) and get
    \[D^{a-1}:=e_{a-1}e_{a-2}D^a=e_{a-2}e_{a-1}D^a.\]
    which shows that $D^{a-1}$ admits $f_{a-2}$.
    Furthermore, if $e_{a-2}$ moved the heavy tile $D^a_{a-1,c_{a-1}}$ to $D^{a-1}_{a-2,c_{a-2}}$ and $e_{a-1}$ moved $D^a_{a,d_a}$ to $D^{a-1}_{a-1,d_{a-1}}$, then $c_{a-1}<d_{a-1}$.
     Keep applying  Lemma~\ref{lem:e-commute}(2) and continue the same reasoning, eventually we can conclude that $D^{i+1}=\Psi(\uparrow B)$ admits $f_i$, contains a single heavy tile in row $i+1$ with column index larger than the $f$-target in row $i$. Finally by Lemma~\ref{lem:e-commute}(2) again we conclude that $\Psi(\uparrow B)$ admits $e_i$. 
\end{proof}
\subsection{Proof of the main theorem}
Equipped with the well-definedness of $\Phi$ and $\Psi$, we are ready to prove our main theorem.
\begin{proof}[Proof of Theorem~\ref{T: main}]
To show $\Psi\circ\Phi=\id$, we proceed by induction on the number of $F$-operators that can be applied on $D\in\MBPD(n)$. The base case is apparent. 
    If $D\neq D_{\id}$ is $F$-terminal, then 
   \[
        \Psi(\Phi(D)) = \Psi(I_i\Phi(\fs_i(D))) 
        =\es_i(\Psi(\Phi(\fs_i(D)))) 
        =\es_i(\fs_i(D)) 
        = D.
    \]
    If $D\neq D_{\id}$ is $F$-nonterminal, then
    \begin{align*}
        \Psi(\Phi(D)) &= \Psi(\downarrow\Phi(f_i(D)))=e_i(\Psi(\uparrow\downarrow\Phi(f_i(D))))\\
        &=e_i(\Psi(\Phi(f_i(D)))) 
        =e_i(f_i(D)) 
        = D.
        \end{align*}
    Similarly, to show $\Phi\circ\Psi=\id$, we proceed by induction on the number of $\XX$ and $\uparrow$ operators that can be applied on $B\in\RCP(n)$.
    If $B=((i,i),(i_2,a_2),\dots)\in\RCP(n)$, then
    \[\Phi(\Psi(B))=\Phi(\es_i(\Psi(X_i(B))))=I_i(\Phi(\fs_i\es_i(\Psi(X_i(B)))))=B.\]
    If $B=((i_1,a_1),(i_2,a_2),\dots)\in\RCP(n)$, with $i_1<a_1$, then
    \[\Phi(\Psi(B))=\Phi(e_i(\Psi(\uparrow(B))))=\downarrow(\Phi(f_ie_i(\Psi(\uparrow B))))=B.\]
    It then follows that $\Phi$ is indeed a bijection.

    We proceed again induction on the number of $F$-operators that can be applied to $D$ to show that $\Phi$ is permutation-preserving. If $D\neq D_{\id}$ is $F$-terminal, then 
    \[w(\Phi(D))=w(\II_i\Phi(\fs_i(D)))=w(\Phi(\fs_i(D)))*s_i=w(\fs_i(D))*s_i=w(D) \]
    where the last step is by Proposition~\ref{P: F permutation}. Finally, if $D\neq D_{\id}$ is $F$-nonterminal, 
    \[w(\Phi(D))=w(\downarrow\Phi(f_i(D)))=w(\Phi(f_i(D)))=w(f_i(D))=w(D).\]
    That $\Psi$ is permutation-preserving follows from it being the inverse of $\Phi$.

    Finally, when restricted to the reduced case, The only relevant cases for $\Phi$ are $\Terminal\Blank$ and $\Ordinary\Blank$, and the corresponding relevant cases for $\Psi$ are $\Initial\Straight$ and $\NPlus\Straight$. 
\end{proof}

\section{An example}
In the following diagrams the red square is the target and the violet square is the opposite corner of the undroop.
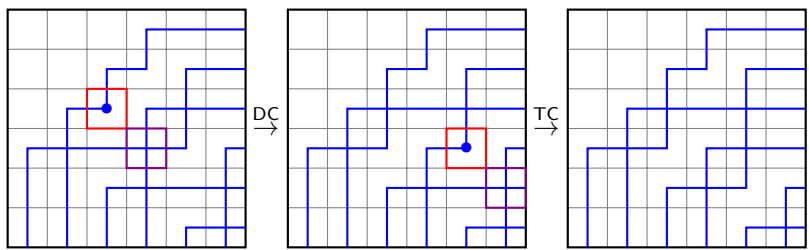
\[
\parbox{1.25in}{
}
\]
The first dual pop produces the pair $(3,4)$.

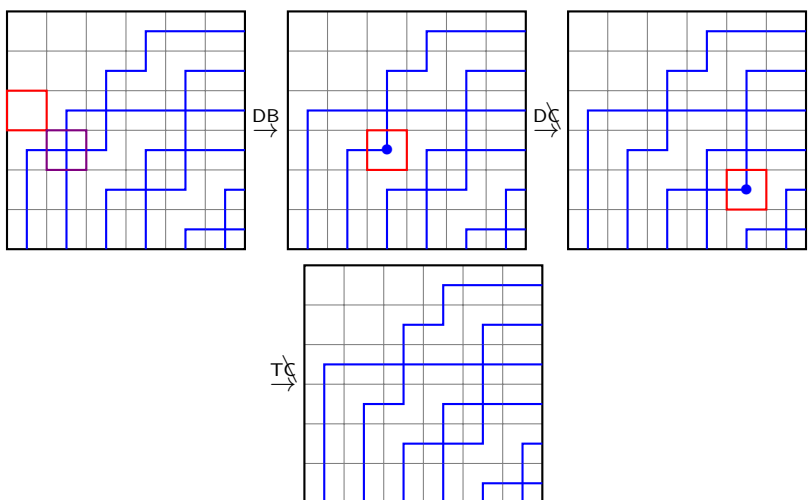
\[
\parbox{1.25in}{%
}
\]
The second dual pop yields $(3,5)$.

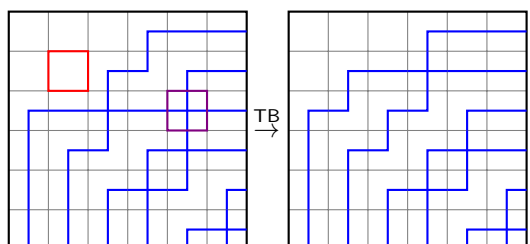
\[
\parbox{1.25in}{%
}
\]
The third dual pop produces $(2,2)$.

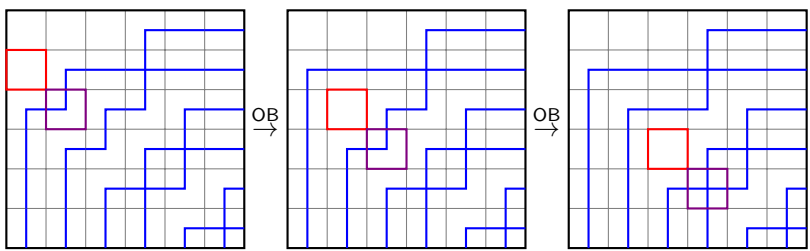
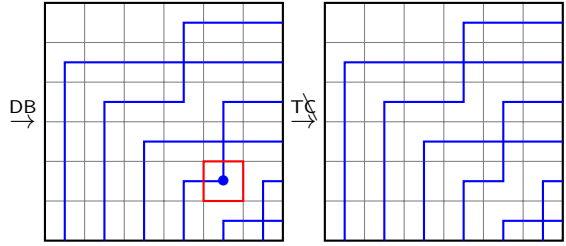
\[
\parbox{1.25in}{
%
}
\]
This dual pop yields $(2,5)$.
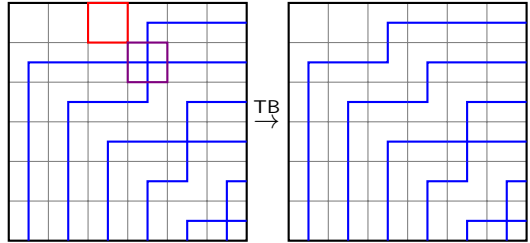
\[
\parbox{1.25in}{
%
}
\]
This dual pop yields $(1,1)$.

We do the entire next dual pop, which consists of three $f$-moves.
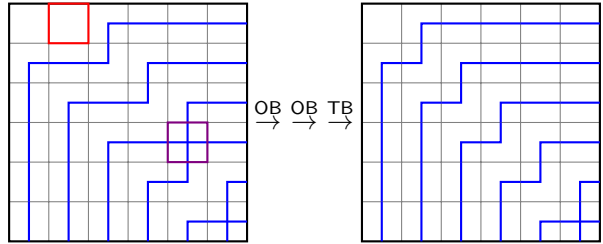
\[
\parbox{1.25in}{%
}
\]
This dual pop yields $(1,3)$.

We do the final dual pop in one step.
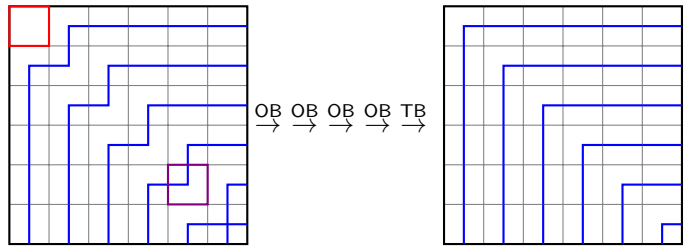
\[
\parbox{1.25in}{%
}
\]
This dual pop yields $(1,5)$.
The final compatible pair is $$B=((3,4),(3,5),(2,2),(2,5),(1,1),(1,3),(1,5)).$$

By Remark \ref{R:CP and PD} $B$ corresponds to the PD having $\ptile$'s at positions
\[((3,2),(3,3),(2,1),(2,4),(1,1),(1,3),(1,5)).\]
\[
%
\]

\bibliographystyle{alphaurl}
\bibliography{ref}
\end{document}